\documentclass[12pt,a4paper]{article}%
\usepackage{amssymb}
\usepackage{amsfonts}
\usepackage{amsmath}
\usepackage{graphicx}%
\providecommand{\U}[1]{\protect\rule{.1in}{.1in}}
\newtheorem{theorem}{Theorem}

\newtheorem{definition}[theorem]{Definition}

\newtheorem{lemma}[theorem]{Lemma}

\newtheorem{proposition}[theorem]{Proposition}

\newenvironment{proof}[1][Proof]{\noindent\textbf{#1.} }{\ \rule{0.5em}{0.5em}}
\makeatletter
\renewcommand\@makefnmark{\relax}
\makeatother
\begin{document}

\title{On the Ornstein-Zernike behaviour for the supercritical Random-Cluster model
on $\mathbb{Z}^{d},d\geq3.$}
\author{M. Campanino$^{\#,\dagger}$, M. Gianfelice$%
{{}^{\circ,\dagger}}%
$\\$^{\#}$Dipartimento di Matematica\\Universit\`{a} degli Studi di Bologna\\P.zza di Porta San Donato, 5 \ I-40127\\massimo.campanino@unibo.it\\$^{%
{{}^\circ}%
}$Dipartimento di Matematica e Informatica\\Universit\`{a} della Calabria\\Campus di Arcavacata\\Ponte P. Bucci - cubo 30B\\I-87036 Arcavacata di Rende\\gianfelice@mat.unical.it}
\maketitle

\begin{abstract}
We prove Ornstein-Zernike behaviour in every direction for finite connection
functions of the random cluster model on $\mathbb{Z}^{d},d\geq3,$ for
$q\geq1,$ when occupation probabilities of the bonds are close to $1.$
Moreover, we prove that equi-decay surfaces are locally analytic, strictly
convex, with positive Gaussian curvature.

\end{abstract}

\footnotetext{
\par
\ \hspace*{0.2cm}\emph{Keywords and phrases}: Random Cluster Model,
Ornstein-Zernike behaviour for connectivities, renormalization, Ruelle
operator, local limit theorem, invariance principle.
\vspace{0.2cm}
\par
\ \hspace*{0.2cm} $^{\dagger}$M. C. and M. G. are partially supported by G.N.A.M.P.A..}

\tableofcontents

\section{Introduction and results}

Ornstein-Zernike behaviour of correlation functions for Gibbs random fields
and of connection functions for percolation models gives an exact power law
correction outside critical points. Apart from its intrinsic interest, in the
two-dimensional case it is related to the behaviour of fluctuations of
interfaces and therefore to the study of phases of two dimensional systems
(\cite{Ga}, \cite{CCC}, \cite{CI}, \cite{CIL}, \cite{CIV2}, \cite{CD-CIV}).

Initially Ornstein-Zernike behaviour has been rigorously proved in the high
temperature/low probability region (see e. g. \cite{BF}). In the last few
decades these results have been extended to subcritical percolation models and
to high temperature finite-range Ising models up to their critical points
(\cite{CCC}, \cite{CI}, \cite{CIV1}, \cite{CIV2}).

Above the critical probability connection functions converge to a positive
constant as the distance of the sites tend to infinity. One is then led to
study the asymptotic behaviour of finite connection functions, i.e. the
probabilities that two sites belong to a common finite open cluster. These
correspond for Gibbs random fields to truncated correlation functions. In
\cite{BF} Bricmont and Fr\"{o}hlich proved Ornstein-Zernike behaviour for
truncated correlation functions of Ising model in the direction of axes in
dimension $d\geq3$ at low tempterature. In the same paper arguments are given,
suggested by their proof, in favour of a different asymptotic behaviour in the
two-dimensional case. A rigorous proof of this in the case of finite
connection functions of two-dimensional Bernoulli percolation above critical
probability is given in \cite{CIL}.

The analysis of the asymptotic behaviour of finite connection functions in
dimension $d\geq3$ has been carried on for Bernoulli percolation with the
parameter close to $1$ in \cite{BPS} for connection functions along Cartesian
axes and then in \cite{CG} for connections in all directions. \cite{BPS} uses
cluster expansions, whereas \cite{CG} exploits the methods developed in
\cite{CI} and \cite{CIV1}, \cite{CIV2}, together with specific techniques
built up to deal with probabilities of non-monotone events such as finite
connections. Here we extend the results of \cite{CG} to FK random cluster
models, with $q\geq1,$ when the probability parameter $p$ is close to $1.$ The
exponential decay of finite connection functions of FK random clusters can be
established by using an inequality proved in \cite{BHK}.

In the rest of this section we present the main results of the paper and the
notation that we will use. In the next section we prove the existence of the
finite correlation length for translation invariant Random Cluster measures
and show that, for $p$ sufficiently close to $1,$ finite supercritical
clusters, up to a negligible probability, have a one-dimensional structure.
This will allow us to reduce the analysis of the exact asymptotics of the
finite two-point connection function to the proof of a local limit theorem
result for an effective stationary random walk via thermodynamic formalism.

\vskip0.5cm

\noindent\textbf{Acknowledgements}. We thank R. van den Berg for pointing out 
the results of reference \cite{BHK} on which Proposition 4 is based and the 
referee for useful comments.

\subsection{Notation}

Given a set $\mathcal{A}\subset\mathbb{R}^{d},d\geq1,$ let us denote by
$\mathcal{A}^{c}$ its complement. We also set $\mathcal{P}\left(
\mathcal{A}\right)  $ to be the collection of all subsets of $\mathcal{A},$
$\mathcal{P}_{n}\left(  \mathcal{A}\right)  :=\{A\in\mathcal{P}\left(
\mathcal{A}\right)  :\left\vert A\right\vert =n\}$ and $\mathcal{P}_{f}\left(
\mathcal{A}\right)  :=\bigcup_{n\geq1}\mathcal{P}_{n}\left(  \mathcal{A}%
\right)  ,$ where $\left\vert A\right\vert $ is the cardinality of $A.$
Moreover, we denote by $\mathcal{\mathring{A}},\overline{\mathcal{A}}$
respectively the interior of $\mathcal{A}$ and the closure of $\mathcal{A}$
and set $\mathfrak{d}\mathcal{A}:=\overline{\mathcal{A}}\backslash
\mathcal{\mathring{A}}$ the boundary of $\mathcal{A}$ in the Euclidean
topology. Furthermore, for $\mathcal{B}\subset\mathbb{R}^{d},$ we set
\begin{equation}
\mathcal{B}+\mathcal{A}:=%
{\displaystyle\bigcup\limits_{x\in\mathcal{B}}}
\left(  x+\mathcal{A}\right)  \ ,
\end{equation}
where, given $x\in\mathbb{R}^{d},$%
\begin{equation}
x+\mathcal{A}:=\left\{  y\in\mathbb{R}^{d}:y-x\in\mathcal{A}\right\}  \ .
\end{equation}

Moreover, we denote by $\left\vert x\right\vert :=\sum_{i=1}^{d}\left\vert
x_{i}\right\vert ,$ by $\left\langle \cdot,\cdot\right\rangle $ the scalar
product in $\mathbb{R}^{d}$ and by $\left\Vert \cdot\right\Vert :=\sqrt
{\left\langle \cdot,\cdot\right\rangle }$ the associated Euclidean norm. We
then set, for $x\neq0,\hat{x}:=\frac{x}{\left\Vert x\right\Vert }%
,\mathbb{S}^{d-1}:=\{z\in\mathbb{R}^{d}:\left\Vert z\right\Vert =1\}$ and,
denoting by $B$ the closed unit ball in $\mathbb{R}^{d},$ for $r>0,$ we let
$rB:=\left\{  x\in\mathbb{R}^{d}:\left\Vert x\right\Vert \leq r\right\}  $ and
$B_{r}\left(  x\right)  :=x+rB.$

For any $t\in\mathbb{R}^{d}$ we define
\begin{equation}
\mathcal{H}^{t}:=\left\{  x\in\mathbb{R}^{d}:\left\langle t,x\right\rangle
=0\right\}
\end{equation}
to be the $\left(  d-1\right)  $-dimensional hyperplane in $\mathbb{R}^{d}$
orthogonal to the vector $t$ passing through the origin and the corresponding
half-spaces
\begin{align}
\mathcal{H}^{t,-}  &  :=\left\{  x\in\mathbb{R}^{d}:\left\langle
t,x\right\rangle \leq0\right\}  \ ,\\
\mathcal{H}^{t,+}  &  :=\left\{  x\in\mathbb{R}^{d}:\left\langle
t,x\right\rangle \geq0\right\}  \ ,
\end{align}
so that, setting for $t\in\mathbb{R}^{d},\mathcal{H}_{x}^{t}:=x+\mathcal{H}%
^{t},$ we denote by $\mathcal{S}_{x,y}^{t}:=\mathcal{H}_{x}^{t,+}%
\cap\mathcal{H}_{y}^{t,-}.$

We also denote by $\mathrm{dist}\left(  \mathcal{A},\mathcal{B}\right)  $ the
Euclidean distance between two subset $\mathcal{A},\mathcal{B}$ of
$\mathbb{R}^{d}.$

\subsubsection{Graphs}

To make the paper self-contained, we will now introduce those notions of graph
theory which are going to be used in the sequel and refer the reader to
\cite{Bo} for an account on this subject.

Let $G=\left(  V,E\right)  $ be a graph whose set of vertices and set of edges
are given respectively by a finite or denumerable set $V$ and $E\subset
\mathcal{P}_{2}\left(  V\right)  .\ G^{\prime}=\left(  V^{\prime},E^{\prime
}\right)  $ such that $V^{\prime}\subseteq V$ and $E^{\prime}\subseteq
\mathcal{P}_{2}\left(  V^{\prime}\right)  \cap E$ is said to be a subgraph of
$G$ and this property is denoted by $G^{\prime}\subseteq G.$ If $G^{\prime
}\subseteq G,$ we denote by $V\left(  G^{\prime}\right)  $ and $E\left(
G^{\prime}\right)  $ respectively the set of vertices and the collection of
the edges of $G^{\prime}.\ \left\vert V\left(  G^{\prime}\right)  \right\vert
$ is called the \emph{order} of $G^{\prime}$ while $\left\vert E\left(
G^{\prime}\right)  \right\vert $ is called its \emph{size}. Given $G_{1}%
,G_{2}\subseteq G,$ we denote by $G_{1}\cup G_{2}:=\left(  V\left(
G_{1}\right)  \cup V\left(  G_{2}\right)  ,E\left(  G_{1}\right)  \cup
E\left(  G_{2}\right)  \right)  \subset G$ the \emph{graph union }of $G_{1}$
and $G_{2}.$ Moreover, we say that $G_{1},G_{2}\subseteq G$ are
\emph{disjoint} if $V\left(  G_{1}\right)  \cap V\left(  G_{2}\right)
=\varnothing.$ A \emph{path} in $G$ is a subgraph $\gamma$ of $G$ such that
there is a bijection $\left\{  0,..,\left\vert E\left(  \gamma\right)
\right\vert \right\}  \ni i\longmapsto v\left(  i\right)  :=x_{i}\in V\left(
\gamma\right)  $ with the property that any $e\in E\left(  \gamma\right)  $
can be represented as $\left\{  x_{i-1},x_{i}\right\}  $ for
$i=1,..,\left\vert E\left(  \gamma\right)  \right\vert .$ A \emph{walk} in $G$
of length $l\geq1$ is an alternating sequence $x_{0},e_{1},x_{1}%
,..,e_{l},x_{l}$ of vertices and edges of $G$ such that $e_{i}=\left\{
x_{i-1},x_{i}\right\}  $ $i=1,..,l.$ Therefore, paths can be associated to
walks having distinct vertices. Two distinct vertices $x,y$ of $G$ are said to
be \emph{connected} if there exists a path $\gamma\subseteq G$ such that
$x_{0}=x,\ x_{\left\vert E\left(  \gamma\right)  \right\vert }=y.$ A graph $G$
is said to be \emph{connected} if any two distinct elements of $V\left(
G\right)  $ are connected. The maximal connected subgraphs of $G$ are called
\emph{components} of $G$ and their number is denoted by $\kappa\left(
G\right)  .$ Moreover, to denote that $\gamma\subset G$ is a component of $G$
we write $\gamma\sqsubset G.$ Given $E^{\prime}\subseteq E,$ we denote by
$G\left(  E^{\prime}\right)  :=\left(  V,E^{\prime}\right)  $ the
\emph{spanning} graph of $E^{\prime}.$ We also define
\begin{equation}
V\left(  E^{\prime}\right)  :=\left(  \bigcup_{e\in E^{\prime}}e\right)
\subset V\ .
\end{equation}
Given $V^{\prime}\subseteq V,$ we set
\begin{equation}
E\left(  V^{\prime}\right)  :=\left\{  e\in E:e\subset V^{\prime}\right\}
\end{equation}
and denote by $G\left[  V^{\prime}\right]  :=\left(  V^{\prime},E\left(
V^{\prime}\right)  \right)  $ that is called the subgraph of $G$
\emph{induced} or \emph{spanned} by $V^{\prime}.$ Moreover, if $G^{\prime
}\subset G,$ we denote by $G\backslash G^{\prime}$ the graph $G\left[
V\backslash V\left(  G^{\prime}\right)  \right]  \subseteq G$ and define the
\emph{boundary} of $G^{\prime}$ as the set
\begin{equation}
\partial G^{\prime}:=\left\{  e\in E\backslash E\left(  G^{\prime}\right)
:\left\vert e\cap V\left(  G^{\prime}\right)  \right\vert =1\right\}  \subset
E\ .\label{defbG}%
\end{equation}

\subsubsection{The Random Cluster measure}

Let $\mathbb{L}^{d}$ denote the graph associated to $\left(  \mathbb{Z}%
^{d},\mathbb{E}^{d}\right)  ,$ with
\begin{equation}
\mathbb{E}^{d}:=\{\{x,y\}\in\mathcal{P}_{2}\left(  \mathbb{Z}^{d}\right)
:\left\vert x-y\right\vert =1\}\ .
\end{equation}
Let $\mathfrak{L}_{0}$ be the collection of subgraphs of $\mathbb{L}^{d}$ of
finite order. If $G\in\mathfrak{L}_{0},$ we denote by $\overline{G}$ the graph
induced by the union of $V\left(  G\right)  $ with the the sets of vertices of
the components of the $\mathbb{L}^{d}\backslash G$ of finite size. We define
the \emph{external boundary} of $G$ to be $\overline{\partial}G:=\partial
\overline{G}.$ We remark that, given $G_{i}:=\left(  V_{i},E_{i}\right)
,\ i=1,2$ two connected subgraphs of $\mathbb{L}^{d}$ of finite size, by
(\ref{defbG}), $\partial\left(  G_{1}\cup G_{2}\right)  \subseteq\partial
G_{1}\cup\partial G_{2}.$ Moreover,
\begin{equation}
\overline{\partial}\left(  G_{1}\cup G_{2}\right)  =\partial\left(
\overline{G_{1}\cup G_{2}}\right)  \subseteq\partial\overline{G_{1}}%
\cup\partial\overline{G_{2}}\ .
\end{equation}
We then set $\mathbb{G}_{0}:=\left\{  G\in\mathfrak{L}_{0}:\partial
G=\overline{\partial}G\right\}  $ and denote by $\mathbb{G}_{c}$ the
collection of connected elements of $\mathbb{G}_{0}.$

Considering the realization of $\mathbb{L}^{d}$ as a geometric graph embedded
in $\mathbb{R}^{d},$ which, with abuse of notation, we still denote by
$\mathbb{L}^{d},$ we can look at it as a cell complex, i.e. as the union of
$\mathbb{Z}^{d}$ and $\mathbb{E}^{d}$ representing respectively the collection
of $0$-cells and of $1$-cells, we denote by $\left(  \mathbb{Z}^{d}\right)
^{\ast}$ the collection of $d$-cells dual to $0$-cells in $\mathbb{L}^{d}%
,$\ that is the collection of Voronoi cells of $\mathbb{L}^{d},$ and by
$\left(  \mathbb{E}^{d}\right)  ^{\ast}$ the collection of $\left(
d-1\right)  $-cells dual $1$-cells in $\mathbb{L}^{d},$ usually called
\emph{plaquettes} in the physics literature.

We also define
\begin{equation}
\mathfrak{E}:=\left\{  \left\{  e_{1}^{\ast},e_{2}^{\ast}\right\}
\in\mathcal{P}_{2}\left(  \left(  \mathbb{E}^{d}\right)  ^{\ast}\right)
:\mathrm{codim}\left(  \mathfrak{d}e_{1}^{\ast}\cap\mathfrak{d}e_{2}^{\ast
}\right)  =2\right\}  \label{defFrakE}%
\end{equation}
and consider the graph $\mathfrak{G}:=\left(  \left(  \mathbb{E}^{d}\right)
^{\ast},\mathfrak{E}\right)  .$

We remark that since duality is an involution: if $E^{\ast}\subset\left(
\mathbb{E}^{d}\right)  ^{\ast},E^{\ast\ast}=E\subset\mathbb{E}^{d}.$

A bond percolation configuration on $\mathbb{L}^{d}$ is a map $\mathbb{E}%
^{d}\ni e\longmapsto\omega_{e}\in\{0,1\}.$ Setting\linebreak$\Omega
:=\{0,1\}^{\mathbb{E}^{d}},$ we define
\begin{equation}
\Omega\ni\omega\longmapsto\mathbf{E}\left(  \omega\right)  :=\left\{
e\in\mathbb{E}^{d}:\omega_{e}=1\right\}  \in\mathcal{P}\left(  \mathbb{E}%
^{d}\right)  \ ,
\end{equation}
Denoting by $\mathbb{G}:=\left\{  G\subseteq\mathbb{L}^{d}:G=G\left(
E\right)  \ ,\ E\in\mathcal{P}\left(  \mathbb{E}^{d}\right)  \right\}  $ the
collection of spanning subgraphs of $\mathbb{L}^{d},$\ we define the random
graph
\begin{equation}
\Omega\ni\omega\longmapsto\mathbf{G}\left(  \omega\right)  :=G\left(
\mathbf{E}\left(  \omega\right)  \right)  \in\mathbb{G}%
\end{equation}
and by $\kappa\left(  \omega\right)  $ the number of its components. Then,
given $l\geq1,\ x_{1},..,x_{l}\in\mathbb{Z}^{d},$ we denote by
\begin{equation}
\Omega\ni\omega\longmapsto\mathbf{C}_{\left\{  x_{1},..,x_{l}\right\}
}\left(  \omega\right)  \in\mathcal{P}\left(  \mathbb{Z}^{d}\right)
\end{equation}
the \emph{common open cluster of the points} $x_{1},..,x_{l}\in\mathbb{Z}%
^{d},$ that is the set of vertices of the component of the random graph
$\mathbf{G}$ to which these points belong, provided it exists, and define, in
the case $\mathbf{C}_{\left\{  x_{1},..,x_{l}\right\}  }$ is finite, the
random set $\overline{\partial}\mathbf{C}_{\{x_{1},..,x_{l}\}}$ to be equal to
$\overline{\partial}G$ if $G$ is the component of $\mathbf{G}$ whose set of
vertices is $\mathbf{C}_{\{x_{1},..,x_{l}\}}$ and the random set
\begin{equation}
\mathbf{S}_{\{x_{1},..,x_{l}\}}:=\left(  \overline{\partial}\mathbf{C}%
_{\{x_{1},..,x_{l}\}}\right)  ^{\ast}\ .
\end{equation}

Let $\mathcal{F}$ be the $\sigma$-algebra generated by the cylinder events of
$\Omega.$ If $\Lambda\subset\subset\mathbb{Z}^{d},$ let $\mathbb{E}^{\Lambda}$
be the subset of $\mathbb{E}^{d}$ such that $V\left(  \mathbb{E}^{\Lambda
}\right)  =\Lambda$ and denote by $\Omega_{\Lambda}:=\left\{  0,1\right\}
^{\mathbb{E}^{\Lambda}},$ by $\mathcal{F}_{\Lambda}$ the corresponding product
$\sigma$-algebra and by $\mathcal{T}_{\Lambda}$ the $\sigma$-algebra generated
by the cylinder events $\left\{  \omega\in\Omega:\omega_{\Delta}\in A\right\}
,$ where $\Delta\subset\Lambda^{c},A\in\mathcal{F}_{\Delta}.$ The Random
Cluster (RC) measures on $\mathbb{Z}^{d}$ (see \cite{FK}, \cite{ES}) with parameters $q\geq1$ and
$\mathbf{p}:=\left\{  p_{e}\right\}  _{e\in\mathbb{E}^{d}},$ where
$\mathbb{E}^{d}\ni e\longmapsto p_{e}\in\left[  0,1\right]  ,$ are
the dependent bond percolation probability measures $\mathbb{P}$ on $\left(
\Omega,\mathcal{F}\right)  $ specified by
\begin{equation}
\mathbb{P}\left(  A|\mathcal{T}_{\Lambda}\right)  =\mathbb{P}_{\Lambda
;q,\mathbf{p}}^{\cdot}\left(  A\right)  \;\mathbb{P}-a.s.\ ,\qquad
A\in\mathcal{F}\ , \label{RCspec}%
\end{equation}
where, setting for any $\pi\in\Omega_{\Lambda}^{c},\Omega_{\Lambda}^{\pi
}:=\left\{  \omega\in\Omega:\omega_{e}=\pi_{e},\ e\in\mathbb{E}^{d}%
\backslash\mathbb{E}^{\Lambda}\right\}  ,\mathbb{P}_{\Lambda;q,\mathbf{p}%
}^{\pi}$ is the probability measure on $\left(  \Omega,\mathcal{F}\right)  $
with density
\begin{equation}
\mathbb{P}_{\Lambda;q,\mathbf{p}}^{\pi}\left(  \omega\right)  :=\frac
{1}{\mathcal{Z}_{\Lambda}^{\pi}\left(  q;\mathbf{p}\right)  }\prod
_{e\in\mathbb{E}^{\Lambda}}p_{e}^{\omega_{e}}\left(  1-p_{e}\right)
^{1-\omega_{e}}q^{\kappa_{\Lambda}\left(  \omega\right)  }\mathbf{1}%
_{\Omega_{\Lambda}^{\pi}}\left(  \omega\right)  \ , \label{P_L}%
\end{equation}
where $\kappa_{\Lambda}\left(  \omega\right)  $ is the number of the
components of $\mathbf{G}\left(  \omega\right)  $ intersecting $\Lambda.$

Random Cluster measures satisfy the FKG inequality, that is, for any couple
$f,g$ of r.v.'s increasing w.r.t. the natural partial order defined on
$\Omega,$ $\mathbb{P}\left(  fg\right)  \geq\mathbb{P}\left(  f\right)
\mathbb{P}\left(  g\right)  .$ Moreover, the partial order of $\Omega$ induces
a stochastic ordering on the elements of the collection of probability
measures defined by (\ref{P_L}); namely, for any increasing r.v.
$f,\ \mathbb{P}_{\Lambda;q,\mathbf{p}}^{\pi_{1}}\left(  f\right)
\leq\mathbb{P}_{\Lambda;q,\mathbf{p}}^{\pi_{2}}\left(  f\right)  $ if $\pi
_{1}\leq\pi_{2}.$ Hence, denoting by $\preceq$ such ordering, $\forall\pi
\in\Omega_{\Lambda}^{c},\ \mathbb{P}_{\Lambda;q,\mathbf{p}}^{\text{f}}%
\preceq\mathbb{P}_{\Lambda;q,\mathbf{p}}^{\pi}\preceq\mathbb{P}_{\Lambda
;q,\mathbf{p}}^{\text{w}},$ where $\mathbb{P}_{\Lambda;q,\mathbf{p}}%
^{\text{f}}$ and $\mathbb{P}_{\Lambda;q,\mathbf{p}}^{\text{w}}$ stand for
respectively the probability measure with density (\ref{P_L}) corresponding to
the \emph{free} ($\pi\equiv0$) and to the \emph{wired} ($\pi\equiv1$) boundary
conditions. Since, for $\#=$f,w, the (weak) limit of the sequence $\left\{
\mathbb{P}_{\Lambda;q,\mathbf{p}}^{\#}\right\}  $ along any exhaustion
$\left\{  \Lambda\right\}  \uparrow\mathbb{Z}^{d}$ exists (see e.g. \cite{Gr}
Theorem 4.19) and is the Random Cluster measure which we denote by
$\mathbb{P}_{q,\mathbf{p}}^{\#},$ the ordering $\preceq$ extends as well to
Random Cluster measures and $\mathbb{P}_{q,\mathbf{p}}^{\text{f}}%
\preceq\mathbb{P}\preceq\mathbb{P}_{q,\mathbf{p}}^{\text{w}}.$

Furthermore, denoting by $\mathbb{P}_{\mathbf{p}^{\prime}}:=\mathbb{P}%
_{1,\mathbf{p}^{\prime}}$ the independent Bernoulli bond percolation
probability measures on $\mathbb{Z}^{d}$ with parameter set $\mathbf{p}%
^{\prime}=\left\{  p_{e}^{\prime}\right\}  _{e\in\mathbb{E}},$ by Theorem
(3.21) p.43 of \cite{Gr}, we obtain the following stochastic domination
inequalities
\begin{equation}
\mathbb{P}_{\mathbf{p}\left(  q\right)  }\preceq\mathbb{P}_{q,\mathbf{p}%
}^{\text{f}}\preceq\mathbb{P}_{q,\mathbf{p}}^{\text{w}}\preceq\mathbb{P}%
_{\mathbf{p}}\ , \label{sdi}%
\end{equation}
where $\forall e\in\mathbb{E}^{d},\ p_{e}\left(  q\right)  :=\frac{p_{e}%
}{p_{e}+q\left(  1-p_{e}\right)  }.$

In the following, we assume the Random Cluster random field specification
defined in (\ref{RCspec}) to be translation invariant; therefore we set,
$\forall e\in\mathbb{E}^{d},p_{e}=p.$ Moreover, we assume the Random Cluster
measure $\mathbb{P}_{q,p}$ to be translation invariant.

\subsection{Results}

\begin{theorem}
\label{thm1}For any $d\geq3$ and any $q\geq1,$ there exists $p_{0}%
=p_{0}\left(  q,d\right)  $ such that, $\forall p>p_{0},$ uniformly in
$x\in\mathbb{Z}^{d}$ as $\left\Vert x\right\Vert \rightarrow\infty,$%
\begin{equation}
\mathbb{P}_{q,p}\left\{  0\longleftrightarrow x\ ,\ \left\vert \mathbf{C}%
_{\{0,x\}}\right\vert <\infty\right\}  =\frac{\Phi_{q,p}\left(  \hat
{x}\right)  }{\sqrt{\left(  2\pi\left\Vert x\right\Vert \right)  ^{d-1}}%
}e^{-\tau_{q,p}\left(  x\right)  }\left(  1+o\left(  1\right)  \right)  \ ,
\end{equation}
where$\ \Phi_{q,p}$ is a positive real analytic function on $\mathbb{S}^{d-1}$
and $\tau_{q,p}$ an equivalent norm in $\mathbb{R}^{d}.$
\end{theorem}

As a by-product of the proof of the previous theorem we also obtain the
following result.

\begin{theorem}
For any $d\geq3$ and any $q\geq1,$ there exists $p_{0}=p_{0}\left(
q,d\right)  $ such that, $\forall p>p_{0},$ the equi-decay set of the
two-point finite connectivity function is locally analytic and strictly
convex. Moreover, the Gaussian curvature of the equi-decay set is uniformly positive.
\end{theorem}

\section{Analysis of connectivities}

Given $x,y\in\mathbb{Z}^{d},$ we set
\begin{equation}
\varphi\left(  x,y\right)  :=\left\{
\begin{array}
[c]{ll}%
\min\left\{  \left\vert \mathbf{S}_{\{x,y\}}\left(  \omega\right)  \right\vert
:\omega\in\left\{  \left\vert \mathbf{C}_{\{x,y\}}\right\vert >0\right\}
\right\}  & x\neq y\\
0 & x=y
\end{array}
\right.  \ . \label{fi}%
\end{equation}
$\varphi$ is symmetric and translation invariant, therefore in the sequel we
will write
\begin{equation}
\varphi\left(  x,y\right)  =\varphi\left(  x-y\right)  \ .
\end{equation}

For any $x\in\mathbb{Z}^{d}$ and $k\geq\varphi\left(  x\right)  ,$ let us set
$\mathbf{A}_{k}\left(  x\right)  :=\left\{  \left\vert \mathbf{S}%
_{\{0,x\}}\left(  \omega\right)  \right\vert =k\right\}  $ and $\mathbf{A}%
^{k}\left(  x\right)  :=%
{\textstyle\bigvee_{l\geq k}}
\mathbf{A}_{l}\left(  x\right)  .$ We define
\begin{equation}
\psi_{k}\left(  x\right)  :=\min\left\{  \left\vert E\left(  \mathbf{C}%
_{\{0,x\}}\left(  \omega\right)  \right)  \right\vert :\omega\in\mathbf{A}%
_{k}\left(  x\right)  \right\}  \ , \label{psikx}%
\end{equation}
and set $\mathbf{A}\left(  x\right)  :=\mathbf{A}_{\varphi\left(  x\right)
}\left(  x\right)  $ and consequently $\psi\left(  x\right)  :=\psi
_{\varphi\left(  x\right)  }.$

By Lemma 6 in \cite{CG} it follows that that there exists $c_{2}=c_{2}\left(
d\right)  >1$ such that, for any $x\in\mathbb{Z}^{d},$%
\begin{equation}
c_{2}^{-1}\leq\frac{\varphi\left(  x\right)  }{\psi\left(  x\right)  }\leq
c_{2}\ . \label{fi/psi}%
\end{equation}

\begin{proposition}
\label{P1}There exists a constant $c_{3}=c_{3}\left(  d\right)  >1$ such that,
for any $p\in\left(  p^{\ast},1\right)  ,$ where
\begin{equation}
p^{\ast}=p^{\ast}\left(  q,d\right)  :=\frac{q\left(  1-\frac{1}{c_{3}%
}\right)  }{\frac{1}{c_{3}}+q\left(  1-\frac{1}{c_{3}}\right)  }\ , \label{p*}%
\end{equation}
and any $\delta>\delta^{\ast},$ with
\begin{gather}
\delta^{\ast}=\delta^{\ast}\left(  p,q,d\right)  :=\frac{\log\frac
{c_{3}\left(  d\right)  q\left(  p+q\left(  1-p\right)  \right)
^{c_{2}\left(  d\right)  -1}}{p^{c_{2}\left(  d\right)  }}}{\log\frac{\left(
p+q\left(  1-p\right)  \right)  }{c_{3}\left(  d\right)  \left(  1-p\right)
q}}\ ,\label{d*}\\
\mathbb{P}_{q,p}\left(  \{\left\vert \mathbf{S}_{\{0,x\}}\right\vert
\geq\left(  1+\delta\right)  \varphi\left(  x\right)  \}|\left\{  0<\left\vert
\mathbf{C}_{\{0,x\}}\right\vert <\infty\right\}  \right) \label{l1}\\
\leq\frac{1}{1-c_{3}\left(  \frac{q\left(  1-p\right)  }{p+q(1-p)}\right)
}\left[  \frac{c_{3}^{1+\delta}\left(  \frac{q}{p+q\left(  1-p\right)
}\right)  ^{1+\delta}\left(  1-p\right)  ^{\delta}}{p^{c_{2}}}\right]
^{\varphi\left(  x\right)  }\ .\nonumber
\end{gather}

\end{proposition}

\begin{proof}
For any $k\geq2d,$ we define the (possibly empty) collection of subgraphs of
$\mathfrak{G}$
\begin{equation}
\mathfrak{G}_{k}:=\left\{  \mathcal{G}\subset\mathfrak{G}:\mathcal{G}=G\left[
\left(  \partial G^{\prime}\right)  ^{\ast}\right]  \ ,\ G^{\prime}%
\in\mathbb{G}_{c}^{d}\ ;\ \left\vert V\left(  \mathcal{G}\right)  \right\vert
=k\right\}  \ .
\end{equation}
We have
\begin{equation}
\{0<\left\vert \mathbf{C}_{\{0,x\}}\right\vert <\infty\}=%
{\displaystyle\bigvee\limits_{k\geq\varphi\left(  x\right)  }}
\mathbf{A}_{k}\left(  x\right)
\end{equation}
and, for any $E\in\left\{  E^{\prime}\subset\mathbb{E}^{d}:E^{\prime}=\mathbf{E}\left(
\omega\right)  ,\omega\in\mathbf{A}_{k}\left(  x\right)  \right\}  ,$ denoting
by
\begin{equation}
\mathbf{A}_{k}\left(  E;x\right)  :=\left\{  \omega\in\mathbf{A}_{k}\left(
x\right)  :\mathbf{E}\left(  \omega\right)  =E\right\}  ,
\end{equation}
we get
\begin{align}
\mathbb{P}_{q,p}\left(  \mathbf{A}_{k}\left(  x\right)  \right)   &
=\mathbb{P}_{q,p}\left(  \left\{  \omega\in\Omega:\left\vert \mathbf{S}%
_{\{0,x\}}\left(  \omega\right)  \right\vert =k\right\}  |\mathbf{A}%
_{k}\left(  E;x\right)  \right)  \mathbb{P}_{q,p}\left(  \mathbf{A}_{k}\left(
E;x\right)  \right)  \\
&  \leq\mathbb{P}_{q,p}\left(  \left\{  \omega\in\Omega:\left\vert
\mathbf{S}_{\{0,x\}}\left(  \omega\right)  \right\vert =k\right\}
|\mathbf{A}_{k}\left(  E;x\right)  \right)  \ .\nonumber
\end{align}
Moreover, because $\left\{  \left\vert \mathbf{S}_{\{0,x\}}\right\vert
=k\right\}  $ is a decreasing event, it holds, since (\ref{sdi}) is also valid
for $\mathbb{P}_{q,p}\left(  \cdot|\mathbf{A}_{k}\left(  E;x\right)  \right)
$ (see \cite{Gr} Theorem (3.1) p.37), that%
\begin{align}
\mathbb{P}_{q,p}\left(  \left\{  \left\vert \mathbf{S}_{\{0,x\}}\right\vert
=k\right\}  |\mathbf{A}_{k}\left(  E;x\right)  \right)   &  \leq
\mathbb{P}_{\frac{p}{p+q(1-p)}}\left\{  \left\vert \mathbf{S}_{\{0,x\}}%
\right\vert =k|\mathbf{A}_{k}\left(  E;x\right)  \right\}  \\
&  \leq\left(  \frac{q\left(  1-p\right)  }{p+q(1-p)}\right)  ^{k}%
\sum_{\mathcal{G}\in\mathbb{G}_{k}}\mathbb{P}_{\frac{p}{p+q(1-p)}}\left\{
G\left[  \mathbf{S}_{\{0,x\}}\right]  =\mathcal{G}\right\}  \nonumber\\
&  \leq\left\vert \mathbb{G}_{k}\right\vert \left(  \frac{q\left(  1-p\right)
}{p+q(1-p)}\right)  ^{k}\ .\nonumber
\end{align}
We can choose for each $\mathcal{G}\in\mathfrak{G}_{k}$ a minimal spanning
tree $T_{\mathcal{G}}$ and consider the collection of graphs
\begin{equation}
\mathfrak{T}_{k}:=\left\{  T_{\mathcal{G}}:\mathcal{G}\in\mathfrak{G}%
_{k}\right\}  \ .
\end{equation}
Since given a connected tree there is a walk passing only twice through any
edge of the graph, there exists a constant $c_{3}=c_{3}\left(  d\right)  >1$
such that $\left\vert \mathfrak{G}_{k}\right\vert =c_{3}^{k}.$ Therefore,
\begin{gather}
\mathbb{P}_{q,p}\left\{  \left\vert \mathbf{S}_{\{0,x\}}\right\vert
\geq\left(  1+\delta\right)  \varphi\left(  x\right)  ,0<\left\vert
\mathbf{C}_{\{0,x\}}\right\vert <\infty\right\}  \leq\mathbb{P}_{q,p}\left\{
\left\vert \mathbf{S}_{\{0,x\}}\right\vert \geq\left(  1+\delta\right)
\varphi\left(  x\right)  \right\}  \\
\leq\sum_{k\geq\left(  1+\delta\right)  \varphi\left(  x\right)  }%
\mathbb{P}_{q,p}\left(  \mathbf{A}_{k}\left(  x\right)  \right)  \leq
\sum_{k\geq\left(  1+\delta\right)  \varphi\left(  x\right)  }c_{3}^{k}\left(
\frac{q\left(  1-p\right)  }{p+q(1-p)}\right)  ^{k}\ .\nonumber
\end{gather}
Since
\begin{equation}
\mathbb{P}_{q,p}\{0<\left\vert \mathbf{C}_{\{0,x\}}\right\vert <\infty
\}=\sum_{k\geq\varphi\left(  x\right)  }\mathbb{P}_{q,p}\left(  \mathbf{A}%
_{k}\left(  x\right)  \right)  \geq\mathbb{P}_{q,p}\left(  \mathbf{A}\left(
x\right)  \right)  \ ,\label{lbP0->x1}%
\end{equation}
for any $E\in\left\{  E^{\prime}\subset\mathbb{E}^{d}:\left\vert E^{\prime
}\right\vert =\psi\left(  x\right)  ,E^{\prime}=\mathbf{E}\left(  \omega\right)
,\omega\in\mathbf{A}\left(  x\right)  \right\}  ,$ denoting by $\mathbf{A}%
\left(  E;x\right)  :=\left\{  \omega\in\mathbf{A}\left(  x\right)  :\mathbf{E}\left(
\omega\right)  =E\right\}  ,$ by (\ref{sdi}) and (\ref{fi/psi}),
\begin{align}
\mathbb{P}_{q,p}\left(  \mathbf{A}\left(  x\right)  \right)   &
=\mathbb{P}_{q,p}\left(  \mathbf{A}\left(  x\right)  |\mathbf{A}\left(
E;x\right)  \right)  \mathbb{P}_{q,p}\left(  \mathbf{A}\left(  E;x\right)
\right)  \label{lbP0->x2}\\
&  \geq\mathbb{P}_{q,p}\left(  \left\{  \left\vert \mathbf{S}_{\{0,x\}}%
\right\vert =\varphi\left(  x\right)  \right\}  |\mathbf{A}\left(  E;x\right)
\right)  \mathbb{P}_{\frac{p}{p+q\left(  1-p\right)  }}\left(  \mathbf{A}%
\left(  E;x\right)  \right)  \nonumber\\
&  \geq\mathbb{P}_{p}\left(  \left\{  \left\vert \mathbf{S}_{\{0,x\}}%
\right\vert =\varphi\left(  x\right)  \right\}  |\mathbf{A}\left(  E;x\right)
\right)  \left(  \frac{p}{p+q\left(  1-p\right)  }\right)  ^{\psi\left(
x\right)  }\nonumber\\
&  \geq\left(  \frac{p}{p+q\left(  1-p\right)  }\right)  ^{\psi\left(
x\right)  }\left(  1-p\right)  ^{\varphi\left(  x\right)  }\nonumber\\
&  \geq\left\{  \left[  \left(  \frac{p}{p+q\left(  1-p\right)  }\right)
\right]  ^{c_{2}}\left(  1-p\right)  \right\}  ^{\varphi\left(  x\right)
}\ .\nonumber
\end{align}
Therefore, $\forall p\in\left(  \frac{q\left(  1-\frac{1}{c_{3}}\right)
}{\frac{1}{c_{3}}+q\left(  1-\frac{1}{c_{3}}\right)  },1\right)  ,$ choosing
$\delta^{\ast}$ as in (\ref{d*}), $\forall\delta>\delta^{\ast},$ we have
\begin{gather}
\mathbb{P}_{q,p}\left(  \{\left\vert \mathbf{S}_{\{0,x\}}\right\vert
\geq\left(  1+\delta\right)  \varphi\left(  x\right)  \}|\left\{  0<\left\vert
\mathbf{C}_{\{0,x\}}\right\vert <\infty\right\}  \right)  \\
\leq\frac{1}{1-c_{3}\left(  \frac{q\left(  1-p\right)  }{p+q(1-p)}\right)
}\frac{\left[  c_{3}\left(  \frac{q\left(  1-p\right)  }{p+q(1-p)}\right)
\right]  ^{\left(  1+\delta\right)  \varphi\left(  x\right)  }}{\left[
\left(  \frac{p}{p+q\left(  1-p\right)  }\right)  ^{c_{2}}\left(  1-p\right)
\right]  ^{\varphi\left(  x\right)  }}\nonumber\\
=\frac{1}{1-c_{3}\left(  \frac{q\left(  1-p\right)  }{p+q(1-p)}\right)
}\left[  \frac{c_{3}^{1+\delta}\left(  \frac{q}{p+q\left(  1-p\right)
}\right)  ^{1+\delta}\left(  1-p\right)  ^{\delta}}{\left(  \frac
{p}{p+q\left(  1-p\right)  }\right)  ^{c_{2}}}\right]  ^{\varphi\left(
x\right)  }\ .\nonumber
\end{gather}

\end{proof}

\begin{proposition}
Given $q\geq1$ and $p\in\left(  0,1\right)  $ let $\mathbb{P}_{q,p}$ be a
translation invariant Random Cluster measure on $\mathbb{L}^{d}$ with
parameters $q$ and $p.$ Then, for any $x\in\mathbb{R}^{d},$%
\begin{equation}
\tau_{q,p}\left(  x\right)  :=\lim_{n\rightarrow\infty}\frac{1}{n}%
\log\mathbb{P}_{q,p}\left\{  0\longleftrightarrow\left\lfloor nx\right\rfloor
\ ,\ \left\vert \mathbf{C}_{\{0,\left\lfloor nx\right\rfloor \}}\right\vert
<\infty\right\}
\end{equation}
exists and is a convex and homogeneous-of-order-one function on $\mathbb{R}%
^{d}.$
\end{proposition}

\begin{proof}
For any $\Delta\subseteq\mathbb{Z}^{d},$ let us denote by $\mathbf{E}_{\Delta
}:=\bigcup_{x\in\Delta}E\left(  \mathbf{C}_{\left\{  x\right\}  }\right)
\subseteq\mathbb{E}^{d}$ the set of edges belonging to open paths starting at
the vertices of $\Delta.$

Let now $\Lambda$ be a finite subset of $\mathbb{Z}^{d}$ such that $\Lambda
\ni0.$ For any two distinct lattice points $x,y\in\Lambda,$ looking at
$\mathbf{1}_{\left\{  0\longleftrightarrow x\ ,\ 0\nleftrightarrow\Lambda
^{c}\right\}  },\mathbf{1}_{\left\{  x\longleftrightarrow
y\ ,\ y\nleftrightarrow\Lambda^{c}\right\}  }$ as functions of $\left(
\mathbf{E}_{\left\{  x\right\}  },\mathbf{E}_{\Lambda^{c}}\right)  ,$ they are
both nondecreasing on $\mathbf{E}_{\left\{  x\right\}  }$ and nonincreasing on
$\mathbf{E}_{\Lambda^{c}}.$ Therefore, by Theorem 2.1 in \cite{BHK},
\begin{gather}
\mathbb{P}_{q,p}\left(  \left\{  0\longleftrightarrow x\ ,\ 0\nleftrightarrow
\Lambda^{c}\right\}  \cap\left\{  x\longleftrightarrow y\ ,\ y\nleftrightarrow
\Lambda^{c}\right\}  |\left\{  x\nleftrightarrow\Lambda^{c}\right\}  \right)
\geq\\
\mathbb{P}_{q,p}\left(  \left\{  0\longleftrightarrow x\ ,\ 0\nleftrightarrow
\Lambda^{c}\right\}  |\left\{  x\nleftrightarrow\Lambda^{c}\right\}  \right)
\mathbb{P}_{q,p}\left(  \left\{  x\longleftrightarrow y\ ,\ y\nleftrightarrow
\Lambda^{c}\right\}  |\left\{  x\nleftrightarrow\Lambda^{c}\right\}  \right)
\ ,\nonumber
\end{gather}

that is
\begin{align}
\mathbb{P}_{q,p}\left\{  x\nleftrightarrow\Lambda^{c}\right\}  \mathbb{P}%
_{q,p}\left\{  0\longleftrightarrow x\ ,\ x\longleftrightarrow
y\ ,\ x\nleftrightarrow\Lambda^{c}\right\}   &  \geq\mathbb{P}_{q,p}\left\{
0\longleftrightarrow x\ ,\ x\nleftrightarrow\Lambda^{c}\right\}  \times\\
&  \times\mathbb{P}_{q,p}\left\{  x\longleftrightarrow y\ ,\ x\nleftrightarrow
\Lambda^{c}\right\}  \ ,\nonumber
\end{align}
which implies
\begin{align}
\mathbb{P}_{q,p}\left\{  0\longleftrightarrow x\ ,\ x\longleftrightarrow
y\ ,\ \mathbf{C}_{\left\{  0,x,y\right\}  }\cap\Lambda^{c}=\varnothing
\right\}   &  \geq\mathbb{P}_{q,p}\left\{  0\longleftrightarrow
x\ ,\ \mathbf{C}_{\left\{  0,x\right\}  }\cap\Lambda^{c}=\varnothing\right\}
\times\\
&  \times\mathbb{P}_{q,p}\left\{  x\longleftrightarrow y\ ,\ \mathbf{C}%
_{\left\{  x,y\right\}  }\cap\Lambda^{c}=\varnothing\right\}  \ .\nonumber
\end{align}
But
\begin{equation}
\mathbb{P}_{q,p}\left\{  0\longleftrightarrow y\ ,\ \mathbf{C}_{\left\{
0,y\right\}  }\cap\Lambda^{c}=\varnothing\right\}  \geq\mathbb{P}%
_{q,p}\left\{  0\longleftrightarrow x\ ,\ x\longleftrightarrow
y\ ,\ \mathbf{C}_{\left\{  0,x,y\right\}  }\cap\Lambda^{c}=\varnothing
\right\}  \ ,
\end{equation}
hence
\begin{align}
\mathbb{P}_{q,p}\left\{  0\longleftrightarrow y\ ,\ \mathbf{C}_{\left\{
0,y\right\}  }\cap\Lambda^{c}=\varnothing\right\}   &  \geq\mathbb{P}%
_{q,p}\left\{  0\longleftrightarrow x\ ,\ \mathbf{C}_{\left\{  0,x\right\}
}\cap\Lambda^{c}=\varnothing\right\}  \times\\
&  \times\mathbb{P}_{q,p}\left\{  x\longleftrightarrow y\ ,\ \mathbf{C}%
_{\left\{  x,y\right\}  }\cap\Lambda^{c}=\varnothing\right\}  \ .\nonumber
\end{align}
Taking the limit $\Lambda\uparrow\mathbb{Z}^{d}$ we have
\begin{align}
\mathbb{P}_{q,p}\left(  \left\{  0\longleftrightarrow y\ ,\ \left\vert
\mathbf{C}_{\{0,y\}}\right\vert <\infty\right\}  \right)   &  \geq
\mathbb{P}_{q,p}\left(  \left\{  0\longleftrightarrow x\ ,\ \left\vert
\mathbf{C}_{\{0,x\}}\right\vert <\infty\right\}  \right)  \times\\
&  \times\mathbb{P}_{q,p}\left(  \left\{  x\longleftrightarrow
y\ ,\ \left\vert \mathbf{C}_{\{x,y\}}\right\vert <\infty\right\}  \right)
\ .\nonumber
\end{align}
Proceeding as in the proof of Proposition 15 in \cite{CG} we obtain the thesis.
\end{proof}

\subsection{Effective structure of connectivities}

\subsubsection{Definitions}

Let $t\in\mathbb{R}^{d}.$ Given two points $x,y\in\mathbb{Z}^{d}$ such that
$\left\langle x,t\right\rangle \leq\left\langle y,t\right\rangle ,$ we denote
by $\mathbf{C}_{\{x,y\}}^{t}$ the cluster of $x$ and $y$ inside the strip
$\mathcal{S}_{x,y}^{t},$ if they are connected in the restriction of the
configuration to $\mathcal{S}_{x,y}^{t}.$

Let $u$ be the first of the unit vectors in the direction of the coordinate
axes $u_{1},..,u_{d}$ such that $\left\langle t,u\right\rangle $ is maximal.

\begin{definition}
Given $t\in\mathbb{R}^{d},$ let $x,y\in\mathbb{Z}^{d}$ such that $\left\langle
x,t\right\rangle \leq\left\langle y,t\right\rangle $ be connected inside
$\mathcal{S}_{x,y}^{t}.$ The points $b\in\mathbf{C}_{\{x,y\}}^{t}$ such that:

\begin{enumerate}
\item $\left\langle t,x+u\right\rangle \leq\left\langle t,b\right\rangle
\leq\left\langle t,y-u\right\rangle ;$

\item $\mathbf{C}_{\{x,y\}}^{t}\cap\mathcal{S}_{b-u,b+u}^{t}=\left\{
b-u,b,b+u\right\}  ;$
\end{enumerate}

are said to be $t$\emph{-break points} of $\mathbf{C}_{\{x,y\}}.$ The
collection of such points, which we remark is a totally ordered set with
respect to the scalar product with $t,$ will be denoted by $\mathbf{B}%
^{t}\left(  x,y\right)  .$
\end{definition}

\begin{definition}
Given $t\in\mathbb{R}^{d},$ let $x,y\in\mathbb{Z}^{d}$ such that $\left\langle
x,t\right\rangle \leq\left\langle y,t\right\rangle $ be connected inside
$\mathcal{S}_{x,y}^{t}.$ An edge $\{b,b+u\}$ such that $b,b+u\in\mathbf{B}%
^{t}\left(  x,y\right)  $ is called a $t$\emph{-bond} of $\mathbf{C}%
_{\{x,y\}}.$ The collection of such edges will be denoted by $\mathbf{E}%
^{t}\left(  x,y\right)  ,$ while $\mathbf{B}_{e}^{t}\left(  x,y\right)
\subset\mathbf{B}^{t}\left(  x,y\right)  $ will denote the subcollection of
$t$-break points\emph{\ }$b$ of $\mathbf{C}_{\{x,y\}}$ such that the edge
$\{b,b+u\}\in\mathbf{E}^{t}\left(  x,y\right)  .$
\end{definition}

For any $t\in\mathbb{R}^{d}$ and $\varepsilon\in\left(  0,1\right)  ,$ let
\begin{equation}
\mathcal{C}_{\varepsilon}\left(  t\right)  :=\left\{  x\in\mathbb{R}%
^{d}:\left(  \hat{t},x\right)  \geq\left(  1-\varepsilon\right)  \left\Vert
x\right\Vert \right\}  \ .
\end{equation}

\begin{definition}
\label{d.cpts}Given $t\in\mathbb{R}^{d},$ let $x,y\in\mathbb{Z}^{d}$ such that
$\left\langle x,t\right\rangle \leq\left\langle y,t\right\rangle $ be
connected inside $\mathcal{S}_{x,y}^{t}.$ Then, for any $\varepsilon\in\left(
0,1\right)  :$

\begin{enumerate}
\item $x$ is said to be a $\left(  t,\varepsilon\right)  $\emph{-forward cone
point} if $\mathbf{C}_{\left\{  x,x+u\right\}  }^{t}=\left\{  x,x+u\right\}  $
and $\mathbf{C}_{\{x,y\}}\cap\mathcal{H}_{x}^{t,+}\subset x+\mathcal{C}%
_{\varepsilon}\left(  t\right)  ;$

\item $y$ is said to be a $\left(  t,\varepsilon\right)  $\emph{-backward cone
point} if $\mathbf{C}_{\left\{  y-u,y\right\}  }^{t}=\left\{  y-u,y\right\}  $
and $\mathbf{C}_{\{x,y\}}\cap\mathcal{H}_{x}^{t,-}\subset y-\mathcal{C}%
_{\varepsilon}\left(  t\right)  ;$

\item $z$ is said to be a $\left(  t,\varepsilon\right)  $\emph{-cone point}
if $z\in\mathbf{B}^{t}\left(  x,y\right)  $ and $\mathbf{C}_{\{z,y\}}\subset
z+\mathcal{C}_{\varepsilon}\left(  t\right)  ,\mathbf{C}_{\{x,z\}}\subset
z-\mathcal{C}_{\varepsilon}\left(  t\right)  .$ The collection of $\left(
t,\varepsilon\right)  $\emph{-cone points} is denoted by $\mathbf{K}%
_{\varepsilon}^{t}\left(  x,y\right)  .$
\end{enumerate}
\end{definition}

\begin{definition}
\label{d.irr}Given $t\in\mathbb{R}^{d},$ let $x,y\in\mathbb{Z}^{d}$ such that
$\left\langle x,t\right\rangle \leq\left\langle y,t\right\rangle $ be
connected inside $\mathcal{S}_{x,y}^{t}.$ Then, for any $\varepsilon\in\left(
0,1\right)  :$

\begin{enumerate}
\item $\mathbf{C}_{\{x,y\}}$ is said to be $\left(  t,\varepsilon\right)
$\emph{-forward irreducible} if $x$ is a $\left(  t,\varepsilon\right)
$-forward cone point and $\mathbf{K}_{\varepsilon}^{t}\left(  x+u,y\right)
=\varnothing;$

\item $\mathbf{C}_{\{x,y\}}$ is said to be $\left(  t,\varepsilon\right)
$\emph{-backward irreducible} if $y$ is a $\left(  t,\varepsilon\right)
$-backward cone point and $\mathbf{K}_{\varepsilon}^{t}\left(  x,y-u\right)
=\varnothing;$

\item $\mathbf{C}_{\{x,y\}}$ is said to be $\left(  t,\varepsilon\right)
$\emph{-irreducible} if $x,y\in\mathbf{K}_{\varepsilon}^{t}\left(  x,y\right)
$ and $\mathbf{K}_{\varepsilon}^{t}\left(  x+u,y-u\right)  =\varnothing.$
\end{enumerate}
\end{definition}

Notice that by definition, if $x$ is a $\left(  t,\varepsilon\right)
$-forward cone point, then is also a $\left(  t,\varepsilon^{\prime}\right)
$-forward cone point for any $\varepsilon^{\prime}\in\left(  \varepsilon
,1\right)  .$The same remark also applies to $\left(  t,\varepsilon\right)
$-backward cone points and therefore to $\left(  t,\varepsilon\right)  $-cone
points implying $\mathbf{K}_{\varepsilon}^{t}\left(  x,y\right)
\subseteq\mathbf{K}_{\varepsilon^{\prime}}^{t}\left(  x,y\right)  .$ Hence, if
for $t\in\mathbb{R}^{d}$ and $x,y\in\mathbb{Z}^{d}$ as in Definition
\ref{d.irr}, there exists $\varepsilon\in\left(  0,1\right)  $ such that
$\mathbf{C}_{\left\{  x,y\right\}  }$\ satisfies either condition 1 or 2 or 3
of that definition, then $\mathbf{C}_{\left\{  x,y\right\}  }$ is said to be
respectively $t$\emph{-forward irreducible, }$t$\emph{-backward irreducible,
}$t$\emph{-irreducible} and we denote by $\mathbf{K}^{t}\left(  x,y\right)
:=\bigcup_{\varepsilon\in\left(  0,1\right)  }\mathbf{K}_{\varepsilon}%
^{t}\left(  x,y\right)  $ the collection of $t$\emph{-cone points} of
$\mathbf{C}_{\left\{  x,y\right\}  }$ as well as\linebreak$\mathcal{E}%
^{t}\left(  x,y\right)  :=\left\{  e\in\mathbf{E}^{t}\left(  x,y\right)
:e\subset\mathbf{K}^{t}\left(  x,y\right)  \right\}  .$

\begin{definition}
\label{comp}Given $t\in\mathbb{R}^{d},$ let $x,y\in\mathbb{Z}^{d}$ such that
$\left\langle x,t\right\rangle \leq\left\langle y,t\right\rangle $ be
connected. Two subclusters $\gamma_{1}$ and $\gamma_{2}$ of $\mathbf{C}%
_{\{x,y\}}$ are said to be \emph{compatible}, which condition we denote by
$\gamma_{1}\coprod\gamma_{2},$ if they are connected and there exists
$b\in\mathbf{K}^{t}\left(  x,y\right)  $ such that $\gamma_{1}$ is a
subcluster of $\mathbf{C}_{\{x,b\}}\cap\mathcal{H}_{b}^{t,-}$ containing $b$
and $\gamma_{2}$ is a subcluster of $\mathbf{C}_{\{b+u,y\}}\cap\mathcal{H}%
_{b+u}^{t,+}$ containing $b+u.$

Therefore, two subsets $\mathbf{s}_{1},\mathbf{s}_{2}$ of $\mathbf{S}%
_{\{x,y\}}$ will be called \emph{compatible}, and we will still denote this
condition by $\mathbf{s}_{1}\coprod\mathbf{s}_{2},$ if there exist two
compatible subclusters $\gamma_{1},\gamma_{2}$ of $\mathbf{C}_{\{x,y\}}$ such
that $\mathbf{s}_{i}=\left(  \overline{\partial}\gamma_{i}\right)  ^{\ast}%
\cap\mathbf{S}_{\{x,y\}},i=1,2.$
\end{definition}

\subsubsection{Renormalization}

In Lemma 4 and Proposition 5 in \cite{CG} we proved that $\varphi$ is
subadditive and the sequence $\left\{  \bar{\varphi}_{n}\right\}
_{n\in\mathbb{N}},$ such that $\forall n\in\mathbb{N},\mathbb{R}^{d}\ni
x\longmapsto\bar{\varphi}_{n}\left(  x\right)  :=\frac{\varphi\left(
\left\lfloor nx\right\rfloor \right)  }{n}\in\mathbb{R}^{+},$ converges
pointwise on $\mathbb{R}^{d},$ and uniformly on $\mathbb{S}^{d-1},$ to a
convex, homogeneous-of-order-one function $\bar{\varphi}.$ As in \cite{CG} we
also define
\begin{equation}
\mathcal{W}:=\bigcap_{\hat{x}\in\mathbb{S}^{d-1}}\left\{  w\in\mathbb{R}%
^{d}:\left\langle w,\hat{x}\right\rangle \leq\bar{\varphi}\left(  \hat
{x}\right)  \right\}  \ .
\end{equation}
Given $x\in\mathbb{Z}^{d},$ let $t\in\mathfrak{d}\mathcal{W}\left(  x\right)
:=\left\{  w\in\mathfrak{d}\mathcal{W}:\left\langle w,x\right\rangle
=\bar{\varphi}\left(  x\right)  \right\}  .$

For $N\in\mathbb{N}$ larger than $1,$ let us set $\mathfrak{t}_{N}%
=\mathfrak{t}_{N}\left(  x\right)  :=\left\lfloor \frac{\left\Vert
x\right\Vert }{N}\right\rfloor -1$ and
\begin{align}
y_{i}  &  :=\left\lfloor iN\hat{x}\right\rfloor \ ;\ \mathcal{H}_{i}%
^{t}:=\mathcal{H}_{y_{i}}^{t}\ ;\mathcal{H}_{i}^{t,-}:=\mathcal{H}_{y_{i}%
}^{t,-}\ ;\ \mathcal{H}_{i}^{t,+}:=\mathcal{H}_{y_{i}}^{t,+}%
\ ,\ i=0,..,\mathfrak{t}_{N}\ ;\\
y_{\mathfrak{t}_{N}+1}  &  :=x\ ;\ \mathcal{H}_{y_{\mathfrak{t}_{N}+1}}%
^{t}:=\mathcal{H}_{x}^{t}\ ;\ \mathcal{H}_{y_{\mathfrak{t}_{N}+1}}%
^{t,-}:=\mathcal{H}_{x}^{t,-}\ ;\\
\mathcal{S}_{i}^{t}  &  :=\mathcal{H}_{i}^{t,+}\cap\mathcal{H}_{i+1}^{t,-}\ .
\end{align}
With a slight notational abuse we still denote by $\mathbf{S}_{\{0,x\}}$ its
representation as a hypersurface in $\mathbb{R}^{d}$ and define
\begin{equation}
\mathbf{C}_{i}^{t}:=\mathbf{C}_{\{0,x\}}\cap\mathcal{S}_{i}^{t}\ ;\ \mathbf{S}%
_{i}^{t}:=\mathbf{S}_{\{0,x\}}\cap\mathcal{S}_{i}^{t}\ .
\end{equation}
Hence, $\mathbf{C}_{\{0,x\}}=%
{\textstyle\bigcup_{i=0}^{\mathfrak{t}_{N}}}
\mathbf{C}_{i}^{t}$ and $\mathbf{S}_{\{0,x\}}\cap\mathcal{S}_{0,x}%
^{t}\subseteq%
{\textstyle\bigcup_{i=0}^{\mathfrak{t}_{N}}}
\mathbf{S}_{i}^{t}\ .$

We call \emph{crossing} any connected component $\mathbf{s}$ of $\mathbf{S}%
_{i}^{t}$ such that, denoting by $\mathcal{K}\left(  \mathbf{s}\right)  $ the
compact subset of $\mathcal{S}_{i}^{t}$ whose boundary is $\mathbf{s},$ there
exist $y\in\mathcal{H}_{i}^{t,-}\cap\mathbb{Z}^{d}$ and $y^{\prime}%
\in\mathcal{H}_{i+1}^{t,+}\cap\mathbb{Z}^{d},$ both belonging to
$\mathbf{C}_{\left\{  0,x\right\}  },$ which are connected by an open path in
$\mathbb{L}^{d}\cap\mathcal{K}\left(  \mathbf{s}\right)  .$

We remark that since $\mathbf{C}_{\{0,x\}}$ is connected, the existence of two
crossings in $\mathcal{S}_{i}^{t}$ implies the existence of two disjoint paths
connecting $\mathcal{H}_{i}^{t}$ and $\mathcal{H}_{i+1}^{t}$ while the
converse does not hold true in general.

We say that a slab $\mathcal{S}_{i}^{t}$ is \emph{good} if $\mathbf{S}_{i}%
^{t}$ is connected and made by just a single crossing of size smaller than
twice the minimal one, otherwise we call it \emph{bad}.

In \cite{CG} we proved that, for $q=1,$ for $\left\vert \mathbf{S}_{\left\{
0,x\right\}  }\right\vert \leq\left(  1+\delta\right)  \varphi\left(
x\right)  ,$ with $\delta>\delta^{\ast}$ given in (\ref{d*}), the number of
bad slabs is at most proportional to $\delta\frac{\left\Vert x\right\Vert }%
{N}.$ This is a purely deterministic statement. To make the paper
self-contained we rederive it here.

Given $t\in\mathbb{R}^{d},$ for any $x,y\in\mathbb{Z}^{d}$ such that
$\left\langle t,x\right\rangle \leq\left\langle t,y\right\rangle ,$ in
\cite{CG}, we introduced the function
\begin{equation}
\phi_{t}\left(  x,y\right)  :=\min_{\omega\in\left\{  \mathcal{H}_{x}%
^{t}\longleftrightarrow\mathcal{H}_{y}^{t}\right\}  }\left\vert \left\{
e^{\ast}\in\mathbf{S}_{\left\{  x,y\right\}  }\left(  \omega\right)  :e^{\ast
}\subset\mathcal{S}_{\left\{  x,y\right\}  }^{t}\right\}  \right\vert \ ,
\label{ffi}%
\end{equation}
where
\begin{equation}
\left\{  \mathcal{H}_{x}^{t}\longleftrightarrow\mathcal{H}_{y}^{t}\right\}
:=\bigcup_{\left(  x^{\prime},y^{\prime}\right)  \in\mathcal{H}_{x}^{t,-}%
\cap\mathbb{Z}^{d}\times\mathcal{H}_{y}^{t,+}\cap\mathbb{Z}^{d}}\left\{
\left\vert \mathbf{C}_{\left\{  x^{\prime},y^{\prime}\right\}  }\right\vert
>0\right\}  \ ,
\end{equation}
which, by translation invariance, we can write $\phi_{t}\left(  x,y\right)
=\phi_{t}\left(  y-x\right)  ,$ and proved (see \cite{CG} Lemma 17) that, for
any $x\in\mathbb{R}^{d}$ and $t\in\mathfrak{d}\mathcal{W},\bar{\phi}%
_{t}\left(  x\right)  :=\lim_{n\rightarrow\infty}\frac{\phi_{t}\left(
\left\lfloor nx\right\rfloor \right)  }{n}=\bar{\varphi}\left(  x\right)  .$
Let $\eta$ be the fraction of slabs containing a portion of $\mathbf{S}%
_{\left\{  0,x\right\}  }$ whose size is larger than or equal to twice the
minimal size of a single crossing. Since any crossing is composed by at least
$\phi_{t}\left(  \left\lfloor N\hat{x}\right\rfloor \right)  $ plaquettes, we
have
\begin{equation}
\frac{\left\Vert x\right\Vert }{N}\left(  \eta2\phi_{t}\left(  \left\lfloor
N\hat{x}\right\rfloor \right)  +\left(  1-\eta\right)  \phi_{t}\left(
\left\lfloor N\hat{x}\right\rfloor \right)  \right)  =\frac{\left\Vert
x\right\Vert }{N}\left(  1+\eta\right)  \phi_{t}\left(  \left\lfloor N\hat
{x}\right\rfloor \right)  <\left(  1+\delta\right)  \varphi\left(  x\right)
\ .
\end{equation}
Moreover, given $\epsilon>0,$ there exists $R_{\epsilon}>0$ such that, for any
$x\in\mathbb{Z}^{d}\cap\left(  R_{\epsilon}B\right)  ^{c},$\linebreak%
$\varphi\left(  x\right)  \leq\bar{\varphi}\left(  x\right)  \left(
1+\epsilon\right)  .$ Hence, choosing $N$ sufficiently large such that
$\phi_{t}\left(  \left\lfloor N\hat{x}\right\rfloor \right)  \leq\bar{\phi
}_{t}\left(  N\hat{x}\right)  \left(  1+\epsilon\right)  ,$ since
$t\in\mathfrak{d}\mathcal{W}\left(  x\right)  ,\ \bar{\phi}_{t}\left(
N\hat{x}\right)  =\bar{\varphi}\left(  N\hat{x}\right)  $ and, by the previous
inequality, we get $\eta<\delta.$ Furthermore, since the number of plaquettes
of $\mathbf{S}_{\left\{  0,x\right\}  }$ exceeding $\varphi\left(  x\right)  $
is at most $\delta\varphi\left(  x\right)  ,$ if $\mathbf{S}_{\left\{
0,x\right\}  }$ does not give rise to multiple crossings and $\mathcal{S}%
_{i}^{t}$ is a bad slab, the components of $\mathbf{S}_{i}^{t}$ which are not
crossings must be connected either to $\mathbf{S}_{i-1}^{t}$ or to
$\mathbf{S}_{i+1}^{t}.$ Therefore, the number of bad slabs with a single
crossing is at most the same as the number of such triples of consecutive
slabs, the last being smaller than $\delta\frac{\left\Vert x\right\Vert }{N}.$
Hence, the total number of bad slabs can be at most equal to $2\delta
\frac{\left\Vert x\right\Vert }{N}.$

Let $t\in\mathfrak{d}\mathcal{W}\left(  x\right)  .$ Denoting by $\left\{
v_{i}\right\}  _{i=1}^{d}$ an orthonormal basis of $\mathbb{R}^{d}$ where
$v_{1}=\hat{t}$ and $\left\{  v_{i}\right\}  _{i=2}^{d}$ is any orthonormal
basis of $\mathcal{H}^{t},$ we define, for $i\in0,..,\mathfrak{t}_{N}$ and
$n:=\left(  n_{2},..,n_{d}\right)  \in\mathbb{Z}^{d-1},$%
\begin{equation}
Q_{N}\left(  i,n\right)  :=\mathcal{S}_{i}^{t}\cap%
{\displaystyle\bigcap\limits_{j=2,..,d}}
\left\{  z\in\mathbb{R}^{d}:\left\langle v_{j},z\right\rangle \in\left[
n_{j}N,\left(  n_{j}+1\right)  N\right]  \right\}  \ ,
\end{equation}
which we call $N$\emph{-blocks}.

For any $i=0,..,\mathfrak{t}_{N},$ we define the $N$\emph{-sets}
$\mathbf{D}_{N}^{t}\left(  i\right)  $ to be convex hull in $\mathbb{R}^{d}$
of the $N$-blocks $\left\{  Q_{N}\left(  i,n\right)  \right\}  _{n\in
\mathbb{Z}^{d-1}}$ intersecting $\mathbf{S}_{i}^{t}.$ Denoting by $\left\{
\mathcal{S}_{i_{l}}^{t}\right\}  _{l=1}^{\mathfrak{g}_{N}}$ the set of good
slabs, the corresponding $N$-sets $\mathbf{D}_{N}^{t}\left(  i_{l}\right)  ,$
will be called \emph{good} while the remaining $N$-sets \emph{bad}, while the
set
\begin{equation}
\mathfrak{C}_{N}:=%
{\displaystyle\bigcup\limits_{i=0}^{\mathfrak{t}_{N}}}
\mathbf{D}_{N}^{t}\left(  i\right)
\end{equation}
will be called $N$\emph{-renormalized cluster}.

\begin{definition}
\label{corr_pts}Given $\varepsilon\in\left(  0,1\right)  $ and $l\in\left\{
1,..,\mathfrak{g}_{N}\right\}  ,$ a point of $z\in\mathbf{C}_{i_{l}}^{t}$ is
called a $\left(  t,\varepsilon\right)  $\emph{-correct point} and the
collection of these points is denoted by $\mathbb{K}_{\varepsilon}^{t}\left(
0,x\right)  ,$ if
\begin{equation}
\bigcup_{j=0}^{i_{l}-1}\mathbf{D}_{N}^{t}\left(  j\right)  \subset
z-\mathcal{C}_{\varepsilon}\left(  t\right)  \ ,\ \bigcup_{j=i_{l}%
+1}^{\mathfrak{t}_{N}}\mathbf{D}_{N}^{t}\left(  j\right)  \subset
z+\mathcal{C}_{\varepsilon}\left(  t\right)  \ .
\end{equation}

\end{definition}

Thus, setting $\mathfrak{k}_{N}:=\left\vert \mathbb{K}_{\varepsilon}%
^{t}\left(  0,x\right)  \right\vert ,\mathbb{K}_{\varepsilon}^{t}\left(
0,x\right)  =\left\{  z_{1},..,z_{\mathfrak{k}_{N}}\right\}  .$ Let $\left\{
\mathcal{S}_{i_{k}}^{t}\right\}  _{k=1}^{\mathfrak{k}_{N}}\subseteq\left\{
\mathcal{S}_{i}^{t}\right\}  _{i=0}^{\mathfrak{t}_{N}}$ such that, for any
$k=1,..,\mathfrak{k}_{N},\mathcal{S}_{i_{k}}^{t}\ni z_{k}$ and define
\begin{equation}
\mathcal{D}_{\varepsilon}^{t}\left(  k\right)  :=\left(  \left(
z_{k}+\mathcal{C}_{\varepsilon}\left(  t\right)  \right)  \cap\left(
z_{k+1}-\mathcal{C}_{\varepsilon}\left(  t\right)  \right)  \right)
\ ,k=1,..,\mathfrak{k}_{N}-1\ . \label{defD^t_e}%
\end{equation}
We select among these compact subsets of $\mathbb{R}^{d}$ those containing
$\bigcup_{j=i_{k}+1}^{i_{k+1}-1}\mathbf{D}_{N}^{t}\left(  j\right)  $ and
denote their collection by $\left\{  \mathcal{D}_{\varepsilon}^{t}\left(
k_{m}\right)  \right\}  _{m=1}^{\mathfrak{c}_{N}-1}.$

\begin{lemma}
There exist $\varepsilon\in\left(  0,1\right)  $ sufficiently large and a
positive constant $c_{5}=c_{5}\left(  \delta,\varepsilon\right)  ,$ such that
$\mathfrak{c}_{N}\geq c_{5}\frac{\left\Vert x\right\Vert }{N}.$
\end{lemma}

\begin{proof}
Let $M$ a positive constant to be chosen later. If $\left\{  \mathcal{S}%
_{i_{l}}^{t}\right\}  _{l=1}^{\mathfrak{g}_{N}},$ we set $m_{1}^{+}:=i_{1}$
and
\begin{align}
r_{1}^{+}  &  :=\min\left\{  k\in\left\{  i_{1},..,\mathfrak{t}_{N}+1\right\}
:\sum_{j=i_{1}}^{k}\left\vert \mathbf{S}_{j}^{t}\right\vert >MN\left(
k-i_{1}\right)  \right\}  \ ;\\
m_{i}^{+}  &  :=\min\left\{  r_{i}^{+}+1,..,\mathfrak{t}_{N}+1\right\}
\cap\left\{  i_{1},..,i_{\mathfrak{g}_{N}}\right\}  \ ;\\
r_{i+1}^{+}  &  :=\min\left\{  k\in\left\{  m_{i}^{+},..,\mathfrak{t}%
_{N}+1\right\}  :\sum_{j=m_{i}^{+}}^{k}\left\vert \mathbf{S}_{j}%
^{t}\right\vert >MN\left(  k-m_{i}^{+}\right)  \right\}  \ .
\end{align}
Analogously, we define $m_{1}^{-}:=i_{\mathfrak{g}_{N}}$ and
\begin{align}
r_{1}^{-}  &  :=\max\left\{  k\in\left\{  0,..,i_{\mathfrak{g}_{N}}\right\}
:\sum_{j=k}^{\mathfrak{g}_{N}}\left\vert \mathbf{S}_{j}^{t}\right\vert
>MN\left(  i_{\mathfrak{g}_{N}}-k\right)  \right\}  \ ;\\
m_{i}^{-}  &  :=\max\left\{  0,..,r_{i}^{-}-1\right\}  \cap\left\{
i_{1},..,i_{\mathfrak{g}_{N}}\right\}  \ ;\\
r_{i+1}^{-}  &  :=\max\left\{  k\in\left\{  0,..,m_{i}^{-}\right\}
:\sum_{j=k}^{m_{i}^{-}}\left\vert \mathbf{S}_{j}^{t}\right\vert >MN\left(
m_{i}^{-}-k\right)  \right\}  \ .
\end{align}
Since $\left\vert \mathbf{S}_{\left\{  0,x\right\}  }\right\vert \leq\left(
1+\delta\right)  \varphi\left(  x\right)  $ and since by Remark 3 in \cite{CG}
there exists $c_{+}=c_{+}\left(  d\right)  >1$ such that $\varphi\left(
x\right)  \leq c_{+}\left\Vert x\right\Vert ,$ we get
\begin{equation}
\left(  1+\delta\right)  c_{+}\left\Vert x\right\Vert \geq\sum_{i\geq1}%
\sum_{j=r_{i}^{+}}^{m_{i+1}^{+}-1}\left\vert \mathbf{S}_{j}^{t}\right\vert
\geq\sum_{i\geq1}\left(  m_{i+1}^{+}-1-r_{i}^{+}\right)  MN
\end{equation}
as well as
\begin{equation}
\sum_{i\geq1}\left(  r_{i}^{-}-m_{i+1}^{-}-1\right)  \leq\frac{\left(
1+\delta\right)  c_{+}}{MN}\left\Vert x\right\Vert \ .
\end{equation}
Hence, if we denote by $\mathfrak{r}_{N}^{\pm}$ the number of the slabs
labelled by the $r_{i}^{+}$'s and $r_{i}^{-}$'s respectively and by
$\mathfrak{m}_{N}^{\pm}$ the number of the (good) slabs labelled by the
$m_{i}^{\pm}$'s, we obtain $\mathfrak{r}_{N}^{\pm}\leq\frac{\left(
1+\delta\right)  c_{+}}{MN}\left\Vert x\right\Vert $ that is, choosing
$M>2\left(  1+\delta\right)  c_{+},$%
\begin{equation}
\mathfrak{c}_{N}=\left\vert \left\{  m_{i}^{+}\right\}  _{i=1}^{\mathfrak{m}%
_{N}^{+}}\cap\left\{  m_{i}^{-}\right\}  _{i=1}^{\mathfrak{m}_{N}^{-}%
}\right\vert \geq\left(  1-2\frac{\left(  1+\delta\right)  c_{+}}{M}\right)
\frac{\left\Vert x\right\Vert }{N}\ .
\end{equation}
This implies that we can find $\varepsilon=\varepsilon\left(  \delta,d\right)
$ and, setting $\left\{  m_{i}^{+}\right\}  _{i=1}^{\mathfrak{m}_{N}^{+}}%
\cap\left\{  m_{i}^{-}\right\}  _{i=1}^{\mathfrak{m}_{N}^{-}}=:\left\{
m_{1},..,m_{\mathfrak{c}_{N}}\right\}  ,$ in each $\mathcal{S}_{m_{i}}%
^{t},i=1,..,\mathfrak{c}_{N},$ a $\left(  t,\varepsilon\right)  $-correct
point $z_{i}$ such that $\mathcal{D}_{\varepsilon}^{t}\left(  k\right)
\supset\bigcup_{j=m_{i}+1}^{m_{i+1}-1}\mathbf{D}_{N}^{t}\left(  j\right)  .$
\end{proof}

By construction, for any $k=1,..,\mathfrak{c}_{N},$ each $z_{k}$ belongs to a
given $\mathbf{D}_{N}^{t}\left(  i_{l}\right)  ,l=1,..,\mathfrak{g}_{N}.$
Therefore, if also $z_{k}+u=z_{k+1}\in\mathbf{D}_{N}^{t}\left(  i_{l}\right)
,$ by the finite-energy property of $\mathbb{P}_{q,p},$ the probability that
$\left\{  z_{k},z_{k}+u\right\}  \in\mathcal{E}^{t}\left(  0,x\right)  $ is
bounded below by $\beta=\beta\left(  M,N\right)  >0$ regardless of the
percolation configuration outside the $N$-set $\mathbf{D}_{N}^{t}\left(
i_{l}\right)  .$ Hence, for a fixed $N$-renormalized cluster $\mathfrak{C}%
_{N}$ containing a subset $\left\{  z_{1},..,z_{2L}\right\}  $ of $t$-correct
points such that, $\forall i=1,..,L,z_{2i}=z_{2i-1}+u$ and any pair $\left\{
z_{2i-1},z_{2i-1}+u\right\}  $ belong to distinct $N$-sets $\mathbf{D}_{N}%
^{t}\left(  i_{l}\right)  ,l\in\left\{  1,..,\mathfrak{g}_{N}\right\}  ,$ the
conditional distribution of $\left\vert \mathcal{E}^{t}\left(  0,x\right)
\right\vert $ given $\mathfrak{C}_{N}$ stochastically dominates the binomial
distribution of parameters $2L$ and $\beta.$ Since the number of $t$-correct
points $\mathfrak{k}_{N}\geq c_{5}\frac{\left\Vert x\right\Vert }{N},$ we have:

\begin{proposition}
\label{expt}For any $p\in\left(  p^{\ast},1\right)  $ sufficiently close to
$1$ and $\delta>\delta^{\ast}$ sufficienly small, with $p^{\ast}$ and
$\delta^{\ast}$ given in respectively (\ref{p*}) and (\ref{d*}), uniformly in
$x$ and in $t\in\mathfrak{d}\mathcal{W}\left(  x\right)  ,$ there exist two
positive constants $c_{6}=c_{6}\left(  \delta,\varepsilon\right)  ,c_{7}%
=c_{7}\left(  N,M,p,q,d\right)  $ such that
\begin{equation}
\mathbb{P}_{q,p}\left(  \left\vert \mathcal{E}^{t}\left(  0,x\right)
\right\vert <c_{6}\frac{\left\Vert x\right\Vert }{N}\}|\left\{  0<\left\vert
\mathbf{C}_{\left\{  0,x\right\}  }\right\vert <\infty\right\}  \right)  \leq
e^{-c_{7}\left\Vert x\right\Vert }\ .
\end{equation}

\end{proposition}

\subsubsection{Reduction to a one-dimensional thermodynamics}

Since by Definition \ref{d.cpts} $\mathbf{K}^{t}\left(  0,x\right)
\subseteq\mathbf{B}^{t}\left(  0,x\right)  ,$ it is a totally ordered set with
respect to the scalar product with $t,$ we can relabel the elements of
$\mathbf{K}^{t}\left(  0,x\right)  $ in increasing order and consider
$\mathbf{K}_{e}^{t}\left(  0,x\right)  :=\bigvee_{i\geq1}\left\{
b_{i},b_{i+u}\right\}  \subseteq\mathbf{K}^{t}\left(  0,x\right)  ,$ where
$\forall i\geq1,\left\langle b_{i+1},t\right\rangle >\left\langle
b_{i},t\right\rangle ,\left\{  b_{i},b_{i+u}\right\}  \in\mathcal{E}%
^{t}\left(  0,x\right)  ,$ which implies that $\mathbf{C}_{\left\{
b_{i}+u,b_{i+1}\right\}  }^{t}$ is a $t$-irreducible subcluster of
$\mathbf{C}_{\left\{  0,x\right\}  }.$ Therefore, we have proven that, with
probability larger than $1-e^{-c_{7}\left\Vert x\right\Vert },$ there exists
$\varepsilon=\varepsilon\left(  \delta\right)  \in\left(  0,1\right)  $ such
that, as in the subcritical case \cite{CIV2}, $\mathbf{C}_{\{0,x\}}$ can be
decomposed as a concatenation of $\left(  t,\varepsilon\right)  $-irreducible
compatible subclusters, that is $\mathbf{C}_{\left\{  0,x\right\}  }%
=\gamma^{b}\coprod\gamma_{1}\coprod..\coprod\gamma_{n}\coprod\gamma^{f},$ for
some $n\in\mathbb{N},n\geq\frac{c_{6}}{N}\left\Vert x\right\Vert ,$ where
$\gamma^{b}=\mathbf{C}_{\left\{  0,b_{1}\right\}  }\cap\mathcal{H}_{b_{1}%
}^{t,-}$ is $\left(  t,\varepsilon\right)  $-backward irreducible, $\gamma
^{f}=\mathbf{C}_{\left\{  b_{n+1},x\right\}  }\cap\mathcal{H}_{b_{n}}^{t,+}$
is $\left(  t,\varepsilon\right)  $-forward irreducible and, for
$i=1,..,n,\gamma_{i}=\mathbf{C}_{\left\{  b_{i}+u,b_{i+1}\right\}  }^{t}$ is
$\left(  t,\varepsilon\right)  $-irreducible. From this follows, by Definition
\ref{comp}, that $\mathbf{S}_{\{0,x\}}$ can be decomposed as a concatenation
of compatible subsets, namely
\begin{equation}
\mathbf{S}_{\left\{  0,x\right\}  }=\mathbf{s}^{b}\coprod\mathbf{s}_{1}%
\coprod..\coprod\mathbf{s}_{n}\coprod\mathbf{s}^{f}\ , \label{Sirrdcp}%
\end{equation}
with $\mathbf{s}^{b}=\left(  \overline{\partial}\gamma^{b}\right)  ^{\ast}%
\cap\mathbf{S}_{\left\{  0,x\right\}  },\mathbf{s}^{f}=\left(  \overline
{\partial}\gamma^{f}\right)  ^{\ast}\cap\mathbf{S}_{\left\{  0,x\right\}  }$
and for $i=1,..,n,\mathbf{s}_{i}=\left(  \overline{\partial}\gamma_{i}\right)
^{\ast}\cap\mathbf{S}_{\left\{  0,x\right\}  }.$ The elements of such a
decomposition of $\mathbf{S}_{\left\{  0,x\right\}  }$ will also be called
$t$\emph{-irreducible}.

If $s$ is a realization of the random element $\mathbf{s}_{i},i=1,..,n,$ part
of the just described decomposition of $\mathbf{S}_{\left\{  0,x\right\}  },$
considering the just given representation of $\mathbf{K}_{e}^{t}\left(
0,x\right)  ,$ we define
\begin{align}
i_{-}\left(  s\right)   &  :=\max\left\{  i\geq1:s\subset\mathcal{H}_{b_{i}%
}^{t,+}\ ;\ b_{i}\in\mathbf{K}_{e}^{t}\left(  0,x\right)  \right\}  \ ,\\
i_{+}\left(  s\right)   &  :=\min\left\{  i\geq i_{-}\left(  s\right)
+1:s\subset\mathcal{H}_{b_{i}}^{t,-}\ ;\ b_{i}\in\mathbf{K}_{e}^{t}\left(
0,x\right)  \right\}  \ .
\end{align}
Moreover, if $s$ is a realization of $\mathbf{s}_{b},$ we define
\begin{equation}
i_{+}\left(  s\right)  :=\min\left\{  i\geq1:s\subset\mathcal{H}_{b_{i}}%
^{t,-}\ ;\ b_{i}\in\mathbf{K}_{e}^{t}\left(  0,x\right)  \right\}  \ ,
\end{equation}
and set $b_{-}\left(  s\right)  :=b_{i_{-}\left(  s\right)  },b_{+}\left(
s\right)  :=b_{i_{+}\left(  s\right)  }.$ Clearly, by definition,
$e_{-}\left(  s\right)  :=\left\{  b_{-}\left(  s\right)  ,b_{-}\left(
s\right)  +u\right\}  $ and $e_{+}\left(  s\right)  :=\left\{  b_{+}\left(
s\right)  -u,b_{+}\left(  s\right)  \right\}  $ belong to $\mathcal{E}%
^{t}\left(  0,x\right)  .$ We also set $\Lambda_{s}$ to be the subset $V$ of
$\mathbb{Z}^{d}$ such that $\left(  \partial G\left[  V\right]  \backslash
\mathcal{E}^{t}\left(  0,x\right)  \right)  ^{\ast}=s.$ Hence, defining, for
any realization $s$ of $\mathbf{s}_{i},i=1,..,n,$%
\begin{align}
\Omega_{s}:=  &  \left\{  \omega\in\Omega:\exists\gamma\sqsubset G\left(
\omega\right)  \text{\ s.t.\ }\left(  \overline{\partial}\gamma\right)
^{\ast}\supset s\ ;\ \omega_{e}=1,\right. \\
&  \left.  \forall e\in\left\{  e_{-}\left(  s\right)  ,e_{-}\left(  s\right)
+u\right\}  \bigvee\left\{  e_{+}\left(  s\right)  -u,e_{+}\left(  s\right)
,\right\}  \right\}  \ ,\nonumber
\end{align}
for any realization $s_{b}$ of $\mathbf{s}_{b}$%
\begin{align}
\Omega_{s_{b}}:=  &  \left\{  \omega\in\Omega:\exists\gamma\sqsubset G\left(
\omega\right)  \text{\ s.t.\ }\left(  \overline{\partial}\gamma\right)
^{\ast}\supset s_{b}\ ;\ \omega_{e}=1,\right. \\
&  \left.  \forall e\in\left\{  e_{+}\left(  s_{b}\right)  -u,e_{+}\left(
s_{b}\right)  \right\}  \right\}  \ \nonumber
\end{align}
and for any realization $s_{f}$ of $\mathbf{s}_{f}$%
\begin{align}
\Omega_{s_{f}}:=  &  \left\{  \omega\in\Omega:\exists\gamma\sqsubset G\left(
\omega\right)  \text{\ s.t.\ }\left(  \overline{\partial}\gamma\right)
^{\ast}\supset s_{f}\ ;\ \omega_{e}=1,\right. \\
&  \left.  \forall e\in\left\{  e_{-}\left(  s_{f}\right)  ,e_{-}\left(
s_{f}\right)  +u\right\}  \right\}  \ ,\nonumber
\end{align}
up to factors of order $1+o\left(  e^{-c_{7}\left\Vert x\right\Vert }\right)
,$%
\begin{gather}
\mathbb{P}_{q,p}\left\{  0\longleftrightarrow x\ ,\ \left\vert \mathbf{C}%
_{\{0,x\}}\right\vert <\infty\right\}  =\label{dcmp}\\
\sum_{s_{b}\Supset0}\sum_{s_{f}\Supset x}\sum_{n\geq1}\sum_{\left(
s_{1},..,s_{n}\right)  }^{\ast}\mathbb{P}_{q,p}\left\{  0\longleftrightarrow
x\ ,\ \mathbf{S}_{\left\{  0,x\right\}  }=s_{b}%
{\textstyle\coprod}
s_{1}%
{\textstyle\coprod}
..%
{\textstyle\coprod}
s_{n}%
{\textstyle\coprod}
s_{f}\right\} \nonumber\\
=\sum_{s_{b}\Supset0}\sum_{s_{f}\Supset x}\sum_{n\geq1}\sum_{\left(
s_{1},..,s_{n}\right)  }^{\ast}\mathbb{P}_{q,p}\left(  \Omega_{b}\cap\left(
{\textstyle\bigcap_{i=1}^{n}}
\Omega_{i}\right)  \cap\Omega_{f}\right)  \ ,
\end{gather}
where $\Omega_{\#}:=\Omega_{s_{\#}},\#=b,f,1,..,n,$ and $\sum_{s_{b}\Supset
0},\sum_{s_{f}\Supset x}$ stand respectively for the sum over the elements of
\begin{align}
&  \left\{  s\subset\left(  \mathbb{E}\right)  ^{\ast}:s=\left(
\overline{\partial}\gamma\right)  ^{\ast}\cap\mathbf{S}_{\left\{  0,x\right\}
}\ \text{s.t. }\gamma\ni0\text{ and is }t\text{-backward irreducible}\right\}
\ ,\\
&  \left\{  s\subset\left(  \mathbb{E}\right)  ^{\ast}:s=\left(
\overline{\partial}\gamma\right)  ^{\ast}\cap\mathbf{S}_{\left\{  0,x\right\}
}\ \text{s.t. }\gamma\ni x\text{ and is }t\text{-forward irreducible}\right\}
\ ,
\end{align}
while the last sum is over all the realizations $\left(  s_{1},..,s_{n}%
\right)  $ of the strings $\left(  \mathbf{s}_{1},..,\mathbf{s}_{n}\right)  $
of $t$-irreducible compatible subsets of $\mathbf{S}_{\left\{  0,x\right\}
}.$

\paragraph{Decomposition of probabilities}

Setting, for any $n\geq1,$%
\begin{equation}
\left\{  s_{\#}\right\}  :=\left\{  \omega\in\Omega:\mathbf{s}_{\#}\left(
\omega\right)  =s_{\#}\right\}  \;,\;\#=b,f,1,..,n \label{def_cf_s}%
\end{equation}
and
\begin{align}
\mathbb{P}_{\Lambda_{s_{b}}\cup\left\{  b_{-}\left(  s_{b}\right)  \right\}
;q,p}^{\text{f}}  &  =:\mathbb{P}_{q,p;s_{b}^{\ast}}^{\text{f}}\ ;\ \mathbb{P}%
_{\left\{  b_{-}\left(  s_{f}\right)  \right\}  \cup\Lambda_{s_{f}}%
;q,p}^{\text{f}}=:\mathbb{P}_{q,p;s_{f}^{\ast}}^{\text{f}}\ ,\\
\mathbb{P}_{\left\{  b_{-}\left(  s_{i}\right)  \right\}  \cup\Lambda_{s_{i}%
}\cup\left\{  b_{-}\left(  s_{i}\right)  \right\}  ;q,p}^{\text{f}}  &
=:\mathbb{P}_{q,p;s_{i}^{\ast}}^{\text{f}}\ ,\\
\mathbb{P}_{\Lambda_{s_{b}}\cup\left(
{\textstyle\bigcup_{i=1}^{n}}
\Lambda_{s_{i}}\right)  \cup\Lambda_{s_{f}};q,p}^{\text{f}}  &  =:\mathbb{P}%
_{q,p;\left(  s_{b}%
{\textstyle\coprod}
s_{1}%
{\textstyle\coprod}
..%
{\textstyle\coprod}
s_{n}%
{\textstyle\coprod}
s_{f}\right)  ^{\ast}}^{\text{f}}\ ,
\end{align}
we have
\begin{gather}
\mathbb{P}_{q,p}\left(  \Omega_{b}\cap\left(
{\textstyle\bigcap_{i=1}^{n}}
\Omega_{i}\right)  \cap\Omega_{f}|\left\{  s_{b}\right\}  \cap\left(
{\textstyle\bigcap_{i=1}^{n}}
\left\{  s_{i}\right\}  \right)  \cap\left\{  s_{f}\right\}  \right) \\
=\mathbb{P}_{q,p;\left(  s_{b}%
{\textstyle\coprod}
s_{1}%
{\textstyle\coprod}
..%
{\textstyle\coprod}
s_{n}%
{\textstyle\coprod}
s_{f}\right)  ^{\ast}}^{\text{f}}\left(  \Omega_{b}\cap\left(
{\textstyle\bigcap_{i=1}^{n}}
\Omega_{i}\right)  \cap\Omega_{f}\right) \nonumber\\
=\mathbb{P}_{q,p;s_{b}^{\ast}}^{\text{f}}\left(  \Omega_{b}\right)
\mathbb{P}_{q,p;s_{f}^{\ast}}^{\text{f}}\left(  \Omega_{f}\right)  \prod
_{i=1}^{n}\mathbb{P}_{q,p;s_{i}^{\ast}}^{\text{f}}\left(  \Omega_{i}\right)
\nonumber
\end{gather}
and
\begin{equation}
\mathbb{P}_{q,p}\left(  \Omega_{\#}\right)  =\mathbb{P}_{q,p}\left(
\Omega_{\#}|\left\{  s_{\#}\right\}  \right)  \mathbb{P}_{q,p}\left\{
s_{\#}\right\}  =\mathbb{P}_{q,p;s_{\#}^{\ast}}^{\text{f}}\left(  \Omega
_{\#}\right)  \mathbb{P}_{q,p}\left\{  s_{\#}\right\}  \ ,\ \#=b,f\ .
\end{equation}
Therefore,
\begin{align}
\mathbb{P}_{q,p}\left(  \Omega_{b}\cap\left(
{\textstyle\bigcap_{i=1}^{n}}
\Omega_{i}\right)  \cap\Omega_{f}\right)   &  =\mathbb{P}_{q,p}\left(
\Omega_{b}\right)  \mathbb{P}_{q,p}\left(  \Omega_{f}\right)  \prod_{i=1}%
^{n}\mathbb{P}_{q,p;s_{i}^{\ast}}^{\text{f}}\left(  \Omega_{i}\right)
\times\\
&  \times\frac{\mathbb{P}_{q,p}\left(  \left\{  s_{b}\right\}  \cap\left(
{\textstyle\bigcap_{i=1}^{n}}
\left\{  s_{i}\right\}  \right)  \cap\left\{  s_{f}\right\}  \right)
}{\mathbb{P}_{q,p}\left\{  s_{b}\right\}  \mathbb{P}_{q,p}\left\{
s_{f}\right\}  }\ .\nonumber
\end{align}
Furthermore, the last term in the r.h.s. of the previous formula admits the
equivalent decompositions
\begin{align}
\frac{\mathbb{P}_{q,p}\left(  \left\{  s_{b}\right\}  \cap\left(
{\textstyle\bigcap_{i=1}^{n}}
\left\{  s_{i}\right\}  \right)  \cap\left\{  s_{f}\right\}  \right)
}{\mathbb{P}_{q,p}\left\{  s_{b}\right\}  \mathbb{P}_{q,p}\left\{
s_{f}\right\}  }  &  =\frac{\mathbb{P}_{q,p}\left(  \left\{  s_{b}\right\}
|\left(
{\textstyle\bigcap_{i=1}^{n}}
\left\{  s_{i}\right\}  \right)  \cap\left\{  s_{f}\right\}  \right)
}{\mathbb{P}_{q,p}\left\{  s_{b}\right\}  }\times\label{decPs}\\
&  \times\prod_{j=1}^{n-1}\mathbb{P}_{q,p}\left(  \left\{  s_{j}\right\}
|\left(
{\textstyle\bigcap_{i=j+1}^{n}}
\left\{  s_{i}\right\}  \right)  \cap\left\{  s_{f}\right\}  \right)
\mathbb{P}_{q,p}\left(  \left\{  s_{n}\right\}  |\left\{  s_{f}\right\}
\right) \nonumber\\
&  =\frac{\mathbb{P}_{q,p}\left(  \left\{  s_{f}\right\}  |\left(
{\textstyle\bigcap_{i=1}^{n}}
\left\{  s_{i}\right\}  \right)  \cap\left\{  s_{b}\right\}  \right)
}{\mathbb{P}_{q,p}\left\{  s_{f}\right\}  }\times\\
&  \times\prod_{j=0}^{n-2}\mathbb{P}_{q,p}\left(  \left\{  s_{n-j}\right\}
|\left(
{\textstyle\bigcap_{i=1}^{n-j-1}}
\left\{  s_{i}\right\}  \right)  \cap\left\{  s_{b}\right\}  \right)
\mathbb{P}_{q,p}\left(  \left\{  s_{1}\right\}  |\left\{  s_{b}\right\}
\right)  \ .\nonumber
\end{align}

Once we have fixed $s_{b}$ and $s_{f},$ we choose one of the just defined
representations, say the first, and, for any $n\in\mathbb{N},$ denoting by
$\mathcal{I}_{t}^{n}$ the collection of strings $\left(  s_{1},..,s_{n}%
\right)  $ of $t$-irreducible compatible subsets of $\mathbf{S}_{\left\{
0,x\right\}  },$ we set
\begin{align}
\mathcal{I}_{t}^{n} &  \ni\left(  s_{1},..,s_{n}\right)  \longmapsto g\left(
s_{1},..,s_{n};s_{b},s_{f}\right)  :=\frac{\mathbb{P}_{q,p}\left(  \left\{
s_{b}\right\}  |\left(
{\textstyle\bigcap_{i=1}^{n}}
\left\{  s_{i}\right\}  \right)  \cap\left\{  s_{f}\right\}  \right)
}{\mathbb{P}_{q,p}\left\{  s_{b}\right\}  }\in\left[  0,+\infty\right)  \ ,\\
\mathcal{I}_{t}^{n} &  \ni\left(  s_{1},..,s_{n}\right)  \longmapsto\Xi\left(
s_{1},..,s_{n};s_{f}\right)  :=\log\mathbb{P}_{q,p;s_{1}^{\ast}}^{\text{f}%
}\left(  \Omega_{1}\right)  \mathbb{P}_{q,p}\left(  \left\{  s_{1}\right\}
|\left(
{\textstyle\bigcap_{j=2}^{n}}
\left\{  s_{j}\right\}  \right)  \cap\left\{  s_{f}\right\}  \right)
\in\left(  -\infty,0\right]  \ .
\end{align}

Let $\mathfrak{S}_{t}:=\bigcup_{n\in\mathbb{N}}\mathfrak{I}_{t}^{n},$ where
$\mathfrak{I}_{t}^{n}$ is the set of infinite sequences $\underline
{s}:=\left(  s_{1},...\right)  $ such that the string composed by the first
$n$ symbols appearing in $\underline{s}$ label the elements of $\mathcal{I}%
_{t}^{n},$ while the remaining symbols are fixed to be the empty set. Setting,
for any $\underline{s},\underline{s}^{\prime}\in\mathfrak{S}_{t}$ such that
$\underline{s}\neq\underline{s}^{\prime},\mathbf{i}\left(  \underline
{s},\underline{s}^{\prime}\right)  :=\min\left\{  k\geq1:s_{k}\neq
s_{k}^{\prime}\right\}  $ and, for any complex-valued function $f$ on
$\mathfrak{S}_{t},\mathbf{var}_{k}\left(  f\right)  :=\sup_{\left\{
\underline{s},\underline{s}^{\prime}\in\mathfrak{S}_{t}\ :\ \mathbf{i}\left(
\underline{s},\underline{s}^{\prime}\right)  \geq k\right\}  }\left\vert
f\left(  \underline{s}\right)  -f\left(  \underline{s}^{\prime}\right)
\right\vert ,$ let $\mathfrak{H}_{\theta}$ be the Banach space of real bounded
continuous functions on $\mathfrak{S}_{t}$ which are also uniformly H\"{o}lder
continuous for a given exponent $\theta<1$ endowed with the norm $\left\Vert
\cdot\right\Vert _{\theta}:=\left\Vert \cdot\right\Vert _{\infty}+\sup
_{k\geq2}\frac{\mathbf{var}_{k}\left(  \cdot\right)  }{\theta^{k-1}}.$

In the next subsection we will prove that $g\left(  \cdot;s_{b},s_{f}\right)
$ and $\Xi\left(  \cdot;s_{f}\right)  $ admit a unique extension on
$\mathfrak{H}_{\theta}$ for some $\theta<1$ denoted respectively by
$g_{s_{b},s_{f}}$ and $\Xi_{s_{f}}.$ This will allow us to define the Ruelle's
operator
\begin{equation}
\mathcal{L}_{s_{f}}f\left(  \underline{s}\right)  :=\sum_{s\in\mathcal{I}_{t}%
}e^{\Xi_{s_{f}}\left(  s,\underline{s}\right)  }f\left(  s,\underline
{s}\right)  \ ,\;f\in\mathfrak{H}_{\theta}\ , \label{Rue}%
\end{equation}
where $\mathcal{I}_{t}:=\mathcal{I}_{t}^{1},$ whose largest isolated
eigenvalue is finite and has multiplicity $1,$ since by Proposition
\ref{expt},%
\begin{equation}
\sup_{\underline{s}\in\mathfrak{S}_{t}}\sum_{s\in\mathcal{I}_{t}}e^{\Xi
_{s_{f}}\left(  s,\underline{s}\right)  }<\infty\ .
\end{equation}
Therefore, up to factors of order $1+o\left(  e^{-c_{7}\left\Vert x\right\Vert
}\right)  ,$ by (\ref{dcmp}),
\begin{equation}
\mathbb{P}_{q,p}\left\{  0\longleftrightarrow x\ ,\ \left\vert \mathbf{C}%
_{\{0,x\}}\right\vert <\infty\right\}  =\sum_{s_{b}\Supset0}\sum_{s_{f}\Supset
x}\sum_{n\geq1}\mathbb{P}_{q,p}\left(  \Omega_{b}\right)  \mathbb{P}%
_{q,p}\left(  \Omega_{f}\right)  \left[  \mathcal{L}_{s_{f}}\right]
^{n}g_{s_{b},s_{f}}\left(  \emptyset\right)  \ ,
\end{equation}
where $\emptyset$ stands for the sequence $\left(  \varnothing,...\right)
\in\mathfrak{S}_{t}.$

Let
\begin{equation}
\mathfrak{K}_{q,p}:=\bigcap_{\hat{y}\in\mathbb{S}^{d-1}}\left\{
w\in\mathbb{R}^{d}:\left\langle w,\hat{y}\right\rangle \leq\tau_{q,p}\left(
\hat{y}\right)  \right\}
\end{equation}
the convex body polar with respect to $\mathcal{U}_{q,p}:=\left\{
y\in\mathbb{R}^{d}:\tau_{q,p}\left(  y\right)  \leq1\right\}  .$ Since
$\tau_{q,p}$ and $\bar{\varphi}$ are equivalent norms in $\mathbb{R}^{d},$ if
$v\in\mathfrak{dK}_{q,p}$ is polar to $x$ (i.e. $\left\langle v,x\right\rangle
=\tau_{q,p}\left(  x\right)  $), we can choose $t=t\left(  v\right)  $ as one of the
elements of $\mathfrak{d}\mathcal{W}\left(  x\right)  $ maximizing its scalar
product with $v.$ Notice that, by translation invariance of the RC random
field, we can consider any realization of $\mathbf{S}_{\left\{  0,x\right\}
}$\ as a collection $s_{b},\left(  s_{1},..,s_{n}\right)  ,s_{f}$ of
realizations of its $t$-irreducible components modulo $\mathbb{Z}^{d}$-shift
patched together. Then, for any element $\mathbf{s}_{i},i\geq1$ of the
$t$-irreducible decomposition of $\mathbf{S}_{\left\{  0,x\right\}  }%
$(\ref{Sirrdcp}) we define
\begin{equation}
X\left(  \mathbf{s}_{i}\right)  :=b_{i+1}-b_{i}\ .
\end{equation}
Thus, up to factors of order $1+o\left(  e^{-c_{7}\left\Vert x\right\Vert
}\right)  ,$ we can write
\begin{gather}
e^{\tau_{q,p}\left(  x\right)  }\mathbb{P}_{q,p}\left\{  0\longleftrightarrow
x\ ,\ \left\vert \mathbf{C}_{\{0,x\}}\right\vert <\infty\right\}
=\label{e^tP}\\
=\sum_{y\in\mathcal{H}_{0}^{v,+}\cap\mathcal{H}_{0}^{t,+}}\sum_{z\in
\mathcal{H}_{x}^{v,-}\cap\mathcal{H}_{x}^{t,}}\sum_{s_{b}\Supset-y}\sum
_{s_{f}\Supset x-z}\mathbb{P}_{q,p}\left(  \Omega_{b}\right)  \mathbb{P}%
_{q,p}\left(  \Omega_{f}\right)  e^{\left\langle v,x-\left(  z-y\right)
\right\rangle }\times\nonumber\\
\times\sum_{n\geq1}\sum_{\left(  s_{1},..,s_{n}\right)  \ :\ \sum_{i=1}%
^{n}X\left(  s_{i}\right)  =z-y}^{\ast}\frac{\mathbb{P}_{q,p}\left(
\Omega_{b}\cap\left(
{\textstyle\bigcap_{i=1}^{n}}
\Omega_{i}\right)  \cap\Omega_{f}\right)  }{\mathbb{P}_{q,p}\left(  \Omega
_{b}\right)  \mathbb{P}_{q,p}\left(  \Omega_{f}\right)  }e^{\left\langle
v,z-y\right\rangle }\nonumber\\
=\sum_{y\in\mathcal{H}_{0}^{v,+}\cap\mathcal{H}_{0}^{t,+}}\sum_{z\in
\mathcal{H}_{x}^{v,-}\cap\mathcal{H}_{x}^{t,-}}\sum_{s_{b}\Supset-y}%
\sum_{s_{f}\Supset x-z}\sum_{n\geq1}\mathbb{P}_{q,p}\left(  \Omega_{b}\right)
\mathbb{P}_{q,p}\left(  \Omega_{f}\right)  e^{\left\langle v,x-\left(
z-y\right)  \right\rangle }\left[  \mathcal{L}_{s_{f}}^{v}\right]
^{n}g_{s_{b},s_{f}}\left(  \emptyset\right)  \ ,\nonumber
\end{gather}
where, assuming the shifts of $t$-backward and $t$-forward irreducible
clusters are normalised in such a way that $b_{1}=b_{n+1}=0,\sum_{s_{b}%
\Supset-y},\sum_{s_{f}\Supset x-y}$ now stand respectively for the sum over
the elements of
\begin{align}
&  \left\{  s\subset\left(  \mathbb{E}\right)  ^{\ast}:s=\left(
\overline{\partial}\gamma\right)  ^{\ast}\cap\mathbf{S}_{\left\{  0,x\right\}
}\ \text{s.t. }\gamma\ni-y\text{ and is }t\text{-backward irreducible}%
\right\}  \ ,\\
&  \left\{  s\subset\left(  \mathbb{E}\right)  ^{\ast}:s=\left(
\overline{\partial}\gamma\right)  ^{\ast}\cap\mathbf{S}_{\left\{  0,x\right\}
}\ \text{s.t. }\gamma\ni x-z\text{ and is }t\text{-forward irreducible}%
\right\}
\end{align}
and, $\mathcal{L}_{s_{f}}^{v}$ is the \emph{tilted} Ruelle's operator on
$\mathfrak{H}_{\theta}$ defined, as in (\ref{Rue}), by the potential $\Xi
_{v}:\mathcal{I}_{t}^{n}\longrightarrow\mathbb{R}$ such that
\begin{equation}
\Xi_{v}\left(  s_{1},..,s_{n};s_{f}\right)  :=\log e^{\left\langle v,X\left(
s_{1}\right)  \right\rangle }\mathbb{P}_{q,p;s_{1}^{\ast}}^{\text{f}}\left(
\Omega_{1}\right)  \mathbb{P}_{q,p}\left(  \left\{  s_{1}\right\}  |\left(
{\textstyle\bigcap_{j=2}^{n}}
\left\{  s_{j}\right\}  \right)  \cap\left\{  s_{f}\right\}  \right)  \ .
\end{equation}

We refer the reader to \cite{CIV1} sections 3.2 and 4 for further details on
Ruelle's Perron-Frobenius theorem on countable alphabets.

\paragraph{Polymer expansion for the supercritical Random Cluster model}

A polymer expansion for the supercritical Random Cluster model has already
been set up in \cite{PS} for any $q>0.$ However, in order to prove Proposition
\ref{P2} below, instead of adapting to our purpose the formalism developed in
that work, we find it more convenient to perform the expansion in a form
closer to the one presented in \cite{KP}.

We can look at the elements of the collection of the connected subgraphs of
finite order of $\mathfrak{G}=\left(  \left(  \mathbb{E}^{d}\right)  ^{\ast
},\mathfrak{E}\right)  ,$ where $\mathfrak{E}$ is defined in (\ref{defFrakE}),
as a set of \emph{polymers} which we denote by $\mathsf{S}.$ Two polymers
$\mathsf{s,s}^{\prime}\in\mathsf{S}$ are said to be \emph{compatible}, and we
write $\mathsf{s}\sim\mathsf{s}^{\prime},$ if they are not connected (as
subgraphs of $\mathfrak{G}$), otherwise are said to be \emph{incompatible} and
we write $\mathsf{s}\nsim\mathsf{s}^{\prime}.$ Given $S\subset\mathsf{S},$ we
denote by $\mathfrak{P}\left(  S\right)  $ the collection of the subsets of
$S$ consisting of mutually compatible polymers and call \emph{contours} the
elements of $\mathfrak{P}_{0}\left(  S\right)  :=\left\{  \sigma
\in\mathfrak{P}\left(  S\right)  :\left\vert \sigma\right\vert <\infty
\right\}  .$ We also set $\mathfrak{P}:=\mathfrak{P}\left(  \mathsf{S}\right)
,\mathfrak{P}_{0}:=\mathfrak{P}_{0}\left(  \mathsf{S}\right)  .$ Given
$S\in\mathcal{P}_{f}\left(  \mathsf{S}\right)  ,\mathsf{s}\in\mathsf{S}$ we
write $S\nsim\mathsf{s}$ if there exists $\mathsf{s}^{\prime}\in S$ such that
$\mathsf{s}^{\prime}\nsim\mathsf{s}.$ Moreover, we call $S$ a \emph{polymer
cluster} if it cannot be decomposed as a union of $S_{1},S_{2}\in
\mathcal{P}_{f}\left(  \mathsf{S}\right)  $ such that every pair
$\mathsf{s}_{1}\in S_{1},\mathsf{s}_{2}\in S_{2}$ is compatible. We denote by
$\mathsf{C}\left(  S\right)  $ the collection of polymer clusters in $S$ and
let $\mathsf{C}$ be the collection of polymer clusters in $\mathsf{S}.$

Given $\Lambda\subset\subset\mathbb{Z}^{d},$ we denote by $\mathsf{S}%
_{\Lambda}$ the subset of $\mathsf{S}$ such that, for any $\mathsf{s}%
\in\mathsf{S}_{\Lambda},\mathsf{s}\subset G\left[  \left(  \mathbb{E}%
^{\Lambda}\right)  ^{\ast}\right]  ,$ where we recall that, for any
$\Delta\subset\mathbb{Z}^{d},G\left[  \left(  \mathbb{E}^{\Delta}\right)
^{\ast}\right]  $ is the subgraph of $\mathfrak{G}$ induced by $\left(
\mathbb{E}^{\Delta}\right)  ^{\ast}.$ We also set $E_{S}^{\ast}:=V\left(
\bigcup_{\mathsf{s}\in S}\mathsf{s}\right)  .$ Then, we define $\kappa
_{\text{w}}\left(  S\right)  $ to be the number of the components of $\left(
\mathbb{Z}^{d},\mathbb{E}^{d}\backslash E_{S}\right)  $ and $\kappa_{\text{f}%
}\left(  S\right)  $ to be the number of the components of $\left(
\Lambda,\mathbb{E}^{\Lambda}\backslash E_{S}\right)  .$ Moreover, for any
$\mathsf{s}\in S_{\Lambda},$ we set $\left\Vert \mathsf{s}\right\Vert
_{\#}:=\kappa_{\#}\left(  \mathsf{s}\right)  -1,\#=$f,w.

Let $\mathfrak{P}_{\Lambda}:=\mathfrak{P}\left(  \mathsf{S}_{\Lambda}\right)
.$ We remark that, given $\mathsf{s}\in S_{\Lambda},$ for any $\sigma
\in\mathfrak{P}_{\Lambda}$ such that $\sigma\ni\mathsf{s},\left\Vert
\mathsf{s}\right\Vert _{\#}=\kappa_{\#}\left(  \sigma\right)  -\kappa
_{\#}\left(  \sigma\backslash\mathsf{s}\right)  .$

The function
\begin{equation}
\mathfrak{P}_{\Lambda}\ni\sigma\longmapsto\Psi_{\#}\left(  \sigma\right)
:=\prod_{\mathsf{s}\in\sigma}\left(  \frac{1-p}{p}\right)  ^{\left\vert
\mathsf{s}\right\vert }q^{\left\Vert \mathsf{s}\right\Vert _{\#}}\in
\mathbb{R}^{+}\;,\quad\#=\text{f,w\ ,}%
\end{equation}
where we set $\Psi_{\#}\left(  \varnothing\right)  :=1,$ is called
\emph{activity} of the contour $\sigma.$ Since, $\left\Vert \mathsf{s}%
\right\Vert _{\#}\leq\left\vert \mathsf{s}\right\vert ,$we get
\begin{equation}
\Psi_{\#}\left(  \sigma\right)  \leq\prod_{\mathsf{s}\in\sigma}\left(
\frac{1-p}{p}q\right)  ^{\left\vert \mathsf{s}\right\vert }\ .
\end{equation}
We then define, for any $S\subseteq\mathsf{S}_{\Lambda},$%
\begin{equation}
\mathfrak{Z}_{q,p}^{\#}\left(  S\right)  :=\sum_{\sigma\in\mathfrak{P}\left(
S\right)  }\Psi_{\#}\left(  \sigma\right)  =\sum_{\sigma\in\mathfrak{P}\left(
S\right)  }\prod_{\mathsf{s}\in\sigma}\left(  \frac{1-p}{p}\right)
^{\left\vert \mathsf{s}\right\vert }q^{\left\Vert \mathsf{s}\right\Vert _{\#}%
}\;,\quad\#=\text{f,w.} \label{defZ}%
\end{equation}
Considering for each $\mathsf{s}\in\mathsf{S}_{\Lambda}$ a minimal spanning
tree and bounding their number as in Proposition \ref{P1}, we obtain that we
can choose $c_{8}>0$ such that, for $p\in\left(  p_{0},1\right)  ,$ with
$p_{0}=p_{0}\left(  q,d\right)  :=\frac{1}{1+\frac{e^{c_{8}}}{qc_{3}}%
\frac{c_{8}}{2+c_{8}}},$%
\begin{equation}
\sum_{\mathsf{s}^{\prime}\in\mathsf{S}_{\Lambda}\ :\ \mathsf{s}^{\prime}%
\nsim\mathsf{s}}e^{c_{8}\left\vert \mathsf{s}^{\prime}\right\vert }\left(
\frac{1-p}{p}q\right)  ^{\left\vert \mathsf{s}^{\prime}\right\vert }%
\leq\left\vert \mathsf{s}\right\vert \frac{c_{3}e^{c_{8}}\left(  \frac{1-p}%
{p}q\right)  }{1-c_{3}e^{c_{8}}\left(  \frac{1-p}{p}q\right)  }\leq\frac
{c_{8}}{2}\left\vert \mathsf{s}\right\vert \ .
\end{equation}
Therefore, given $\mathsf{s}\in\mathsf{S}_{\Lambda},$ if $\ell\left(
\mathsf{s}\right)  $ denotes the diameter of $V\left(  \mathsf{s}\right)  $
considered as a subset of $\mathbb{R}^{d},$ since $\ell\left(  \mathsf{s}%
\right)  \leq\left\vert \mathsf{s}\right\vert ,$%
\begin{equation}
\sum_{\mathsf{s}^{\prime}\in\mathsf{S}_{\Lambda}\ :\ \mathsf{s}^{\prime}%
\nsim\mathsf{s}}e^{\frac{c_{8}}{2}\left\vert \mathsf{s}^{\prime}\right\vert
+\frac{c_{8}}{2}\ell\left(  \mathsf{s}^{\prime}\right)  }\left(  \frac{1-p}%
{p}\right)  ^{\left\vert \mathsf{s}^{\prime}\right\vert }q^{\left\Vert
\mathsf{s}^{\prime}\right\Vert _{\#}}\leq\frac{c_{8}}{2}\left\vert
\mathsf{s}\right\vert \ .
\end{equation}
Thus, by the theorem in \cite{KP}, for any $S\subseteq\mathsf{S}_{\Lambda},$%
\begin{equation}
\log\mathfrak{Z}_{q,p}^{\#}\left(  S\right)  =\sum_{S^{\prime}\in
\mathsf{C}\left(  S\right)  }\vartheta_{\#}\left(  S\right)  \;,\quad
\#=\text{f,w} \label{deftheta}%
\end{equation}
where, setting $\mathsf{C}_{\Lambda}:=\mathsf{C}\left(  \mathsf{S}_{\Lambda
}\right)  ,$
\begin{equation}
\mathsf{C}_{\Lambda}\ni S\longmapsto\vartheta_{\#}\left(  S\right)
:=\sum_{S^{\prime}\in\mathcal{P}\left(  S\right)  }\left(  -1\right)
^{\left\vert S\right\vert -\left\vert S^{\prime}\right\vert }\log
\mathfrak{Z}_{q,p}^{\#}\left(  S^{\prime}\right)
\end{equation}
is such that, $\forall\mathsf{s}\in\mathsf{S}_{\Lambda},$%
\begin{equation}
\sum_{S\in\mathsf{C}_{\Lambda}\ :\ S\nsim\mathsf{s}}e^{\frac{c_{8}}{2}%
\sum_{\mathsf{s}^{\prime}\in S}\ell\left(  \mathsf{s}^{\prime}\right)
}\left\vert \vartheta_{\#}\left(  S\right)  \right\vert \leq\frac{c_{8}}%
{2}\left\vert \mathsf{s}\right\vert \ . \label{ext1}%
\end{equation}

Condition (\ref{ext1}) provides the existence of thermodynamics for the
polymer model with partition function $\mathfrak{Z}_{\Lambda}^{\#}\left(
q,p\right)  :=\sum_{\sigma\in\mathfrak{P}_{\Lambda}}\Psi_{\#}\left(
\sigma\right)  ,$ i.e. the existence of the limit $\lim_{\Lambda
\uparrow\mathbb{Z}^{d}}\frac{\log\mathfrak{Z}_{\Lambda}^{\#}\left(
q,p\right)  }{\left\vert \Lambda\right\vert }$ along any cofinal sequence (see
\cite{Ge}) $\left\{  \Lambda\right\}  \uparrow\mathbb{Z}^{d}$ \cite{KP}, this
limit being independent of the boundary conditions.

Considering the realization of the elements of the decomposition of
$\mathbf{S}_{\left\{  0,x\right\}  }$ given in (\ref{Sirrdcp}) as elements of
$\mathsf{S}$ we have

\begin{lemma}
Let $s_{1},s,s_{f},s_{f}^{\prime}$ be realizations of respectively
$\mathbf{s}_{1},\mathbf{s},\mathbf{s}_{f}.$ Then, by (\ref{def_cf_s}) and
(\ref{decPs}) there exists $c_{10}>0$ such that
\begin{equation}
\frac{\mathbb{P}_{q,p}\left(  \left\{  s_{1}\right\}  |\left\{  s\right\}
\cap\left\{  s_{f}\right\}  \right)  }{\mathbb{P}_{q,p}\left(  \left\{
s_{1}\right\}  |\left\{  s\right\}  \cap\left\{  s_{f}^{\prime}\right\}
\right)  }\leq\exp e^{-c_{10}dist\left(  s_{1},s_{f}\triangle s_{f}^{\prime
}\right)  }\ .
\end{equation}

\end{lemma}

\begin{proof}
By (\ref{P_L}), (\ref{defZ}), and (\ref{deftheta}), we obtain
\begin{align}
\frac{\mathbb{P}_{q,p}\left(  \left\{  s_{1}\right\}  |\left\{  s\right\}
\cap\left\{  s_{f}\right\}  \right)  }{\mathbb{P}_{q,p}\left(  \left\{
s_{1}\right\}  |\left\{  s\right\}  \cap\left\{  s_{f}^{\prime}\right\}
\right)  }  &  =\frac{\mathbb{P}_{q,p}\left(  \left\{  s_{1}\right\}
\cap\left\{  s\right\}  \cap\left\{  s_{f}\right\}  \right)  }{\mathbb{P}%
_{q,p}\left(  \left\{  s_{1}\right\}  \cap\left\{  s\right\}  \cap\left\{
s_{f}^{\prime}\right\}  \right)  }\frac{\mathbb{P}_{q,p}\left(  \left\{
s\right\}  \cap\left\{  s_{f}^{\prime}\right\}  \right)  }{\mathbb{P}%
_{q,p}\left(  \left\{  s\right\}  \cap\left\{  s_{f}\right\}  \right)  }\\
&  =\exp\left[  \sum_{S\in\mathsf{C}\ :\ S\nsim s_{1}\coprod s\coprod s_{f}%
}\vartheta_{\text{w}}\left(  S\right)  -\sum_{S\in\mathsf{C}\ :\ S\nsim
s\coprod s_{f}}\vartheta_{\text{w}}\left(  S\right)  +\right. \nonumber\\
&  \left.  +\sum_{S\in\mathsf{C}\ :\ S\nsim s\coprod s_{f}^{\prime}}%
\vartheta_{\text{w}}\left(  S\right)  -\sum_{S\in\mathsf{C}\ :\ S\nsim
s_{1}\coprod s\coprod s_{f}^{\prime}}\vartheta_{\text{w}}\left(  S\right)
\right] \nonumber\\
&  =\exp\left[  \sum_{S\in\mathsf{C}\ :\ S\nsim s_{1},S\nsim s_{f}}%
\vartheta_{\text{w}}\left(  S\right)  -\sum_{S\in\mathsf{C}\ :\ S\nsim
s_{1},S\nsim s_{f}^{\prime}}\vartheta_{\text{w}}\left(  S\right)  \right]
\nonumber\\
&  =\exp\sum_{S\in\mathsf{C}\ :\ S\nsim s_{1},S\nsim s_{f}\triangle
s_{f}^{\prime}}\vartheta_{\text{w}}\left(  S\right)  \ .\nonumber
\end{align}
Since, by definition of $s_{1},s_{f}$ and $s_{f}^{\prime},$ there exists
$\varepsilon>0$ and $b_{+}\left(  s_{1}\right)  ,b_{-}\left(  s_{f}\right)
,b_{-}\left(  s_{f}^{\prime}\right)  \in\mathbb{Z}^{d}$ such that
$s_{1}\subset b_{+}\left(  s_{1}\right)  -\mathcal{C}_{\varepsilon}\left(
t\right)  ,s_{f}\subset b_{-}\left(  s_{f}\right)  +\mathcal{C}_{\varepsilon
}\left(  t\right)  $ and $s_{f}^{\prime}\subset b_{-}\left(  s_{f}^{\prime
}\right)  +\mathcal{C}_{\varepsilon}\left(  t\right)  ,$

for any $\mathsf{s}\in S$ such that $S\in\mathsf{C}$ and $S\nsim s_{1}\bigvee
s_{f},\ell\left(  \mathsf{s}\right)  \geq\left\Vert b_{+}\left(  s_{1}\right)
-b_{-}\left(  s_{f}\right)  \right\Vert .$ Hence, given $\zeta\in\left(
0,1\right)  ,$ let us define, for any $l\geq0,\mathcal{S}_{l\zeta N,b_{2}}%
^{t}:=\mathcal{H}_{b_{2}+\hat{t}l\zeta N}^{t,+}\cap\mathcal{H}_{b_{2}+\hat
{t}\left(  l+1\right)  \zeta N}^{t,-}$ and $\mathcal{S}_{l\zeta N,b_{1}}%
^{t}:=\mathcal{H}_{b_{1}-\hat{t}l\zeta N}^{t,-}\cap\mathcal{H}_{b_{1}-\hat
{t}\left(  l+1\right)  \zeta N}^{t,+},$ where $b_{2}$ is the element of the
set $\left\{  b_{-}\left(  s_{f}\right)  ,b_{-}\left(  s_{f}^{\prime}\right)
\right\}  $ closer to $b_{1}:=b_{+}\left(  s_{1}\right)  $ w.r.t. the
Euclidean distance. Setting $s_{1}^{\left(  l\right)  }:=s_{1}\cap
\mathcal{S}_{l\zeta N,b_{1}}^{t},s_{f}^{\left(  k\right)  }:=\left(
s_{f}\triangle s_{f}^{\prime}\right)  \cap\mathcal{S}_{l\zeta N,b_{1}}^{t},$
there exists $c_{9}=c_{9}\left(  \varepsilon\right)  >0$ such that, by
(\ref{ext1}), we have%
\begin{align}
&  \sum_{S\in\mathsf{C}\ :\ S\nsim s_{1},S\nsim s_{f}\triangle s_{f}^{\prime}%
}\left\vert \vartheta_{\text{w}}\left(  S\right)  \right\vert =\sum_{k,l\geq
0}\sum_{S\in\mathsf{C}\ :\ S\nsim\mathsf{s}\ \text{s.t. }\mathsf{s}\in
s_{1}^{\left(  k\right)  }\bigvee s_{f}^{\left(  l\right)  }}\left\vert
\vartheta_{\text{w}}\left(  S\right)  \right\vert \\
&  \leq\sum_{k,l\geq0}e^{-c_{8}dist\left(  s_{1}^{\left(  k\right)  }%
,s_{f}^{\left(  l\right)  }\right)  }\sum_{S\in\mathsf{C}\ :\ S\nsim
\mathsf{s}\ \text{s.t. }\mathsf{s}\in s_{1}^{\left(  k\right)  }\bigvee
s_{f}^{\left(  l\right)  }}\left\vert \vartheta_{\text{w}}\left(  S\right)
\right\vert e^{c_{8}\sum_{\mathsf{s}^{\prime}\in S}\ell\left(  \mathsf{s}%
^{\prime}\right)  }\nonumber\\
&  \leq e^{-c_{8}\left\Vert b_{1}-b_{2}\right\Vert }\left(  c_{9}\zeta
N\int_{0}^{\infty}dre^{-c_{8}r}r^{d-1}\right)  ^{2}\ .\nonumber
\end{align}

\end{proof}

A straightforward consequence of this result is the following

\begin{proposition}
\label{P2}There exists $\theta=\theta\left(  p,q,d\right)  \in\left(
0,1\right)  $ and two positive constants $c_{11},c_{12}$ such that uniformly
in $v\in\mathfrak{dK}_{q,p},t$-irreducible subsets $s_{1},s_{b},$ strings of
$t$-irreducible subsets $\underline{s},\underline{s}^{\prime},$ and pairs of
$t$-irreducible subsets $s_{f},s_{f}^{\prime}:$%
\begin{gather}
c_{11}\leq g\left(  \underline{s}|s_{b},s_{f}\right)  \leq\frac{1}{c_{11}%
}\ ,\\
\left\vert g\left(  s_{1},\underline{s}|s_{b},s_{f}\right)  -g\left(
s_{1},\underline{s}^{\prime}|s_{b},s_{f}^{\prime}\right)  \right\vert \leq
c_{12}\theta^{\mathbf{i}\left(  \underline{s},\underline{s}^{\prime}\right)
}\ ,\\
\left\vert \Xi_{v}\left(  s_{1},\underline{s}|s_{b},s_{f}\right)  -\Xi
_{v}\left(  s_{1},\underline{s}^{\prime}|s_{b},s_{f}^{\prime}\right)
\right\vert \leq c_{12}\theta^{\mathbf{i}\left(  \underline{s},\underline
{s}^{\prime}\right)  }\ .
\end{gather}

\end{proposition}

\subsection{Exact asymptotics of finite connections}

We refer to \cite{CIV1} section 5 for the derivation of local limit type
results associated with Ruelle's operators on countable alphabets.

Let $p\in\left(  p_{0}\vee p^{\ast},1\right)  .$ For any $n\in\mathbb{N},$ the
measure on $\mathcal{I}_{t}^{n},$%
\begin{equation}
\nu_{n}^{v}\left(  s_{1},..,s_{n}|s_{b},s_{f}\right)  :=e^{\sum_{i=1}^{n}%
\Xi_{v}\left(  s_{i},..,s_{n};s_{f}\right)  }g\left(  s_{1},..,s_{n}%
;s_{b},s_{f}\right)
\end{equation}
allow us to represent (\ref{e^tP}) as
\begin{align}
e^{\tau_{q,p}\left(  x\right)  }\mathbb{P}_{q,p}\left\{  0\longleftrightarrow
x\ ,\ \left\vert \mathbf{C}_{\{0,x\}}\right\vert <\infty\right\}   &
=\sum_{y\in\mathcal{H}_{0}^{v,+}\cap\mathcal{H}_{0}^{t,+}}\sum_{z\in
\mathcal{H}_{x}^{v,-}\cap\mathcal{H}_{x}^{t,-}}\sum_{s_{b}\Supset-y}%
\sum_{s_{f}\Supset x-z}\mathbb{P}_{q,p}\left(  \Omega_{b}\right)
\mathbb{P}_{q,p}\left(  \Omega_{f}\right)  e^{\left\langle v,x-\left(
z-y\right)  \right\rangle }\times\label{e^tPn}\\
&  \times\sum_{n\geq1}\nu_{n}^{v}\left(  \sum_{i=1}^{n}X\left(  s_{i}\right)
=z-y|s_{b},s_{f}\right)  \ .\nonumber
\end{align}
Because Proposition \ref{expt} implies
\begin{equation}
\sum_{s_{b}\Supset-u}\mathbb{P}_{q,p}\left(  \Omega_{b}\right)
e^{\left\langle v,u\right\rangle }\leq e^{-c_{7}\left\Vert u\right\Vert
}\;;\;\sum_{s_{f}\Supset u}\mathbb{P}_{q,p}\left(  \Omega_{f}\right)
e^{\left\langle v,u\right\rangle }\leq e^{-c_{7}\left\Vert u\right\Vert }%
\end{equation}
uniformly in $u\in\mathbb{Z}^{d},$ the main contribution of the r.h.s. of
(\ref{e^tP}) comes from the last sum in (\ref{e^tPn}) when $z-y$ is close to
$x$ and $n$ is close to the optimal value. Therefore, proceding as in section
4.1 of \cite{CIV2} we have
\begin{equation}
\sum_{n\geq1}\nu_{n}^{v}\left(  \sum_{i=1}^{n}X\left(  s_{i}\right)
=z-y|s_{b},s_{f}\right)  =\frac{\Theta_{q,p}\left(  \hat{x}\right)  }%
{\sqrt{\left(  2\pi\left\Vert x\right\Vert \right)  ^{d-1}}}F\left(
s_{b}\right)  F\left(  s_{f}\right)  \left(  1+o\left(  1\right)  \right)  \ ,
\end{equation}
where $\Theta$ is a locally analytic positive function defined on a
neighborhood of $\hat{x}$ in $\mathbb{S}^{d-1}$ and $F$ is a function on the
set of all the possible realizations of $t$-backward and $t$-forward subsets
$s_{b}$ and $s_{f}$ which is bounded above and below.

This proves Theorem \ref{thm1} with%

\begin{equation}
\Phi_{q,p}\left(  \hat{x}\right)  =\Theta_{q,p}\left(  \hat{x}\right)  \left(
\sum_{u\in\mathcal{H}_{0}^{v,+}\cap\mathcal{H}_{0}^{t,+}}\sum_{s_{f}\Supset
u}\mathbb{P}_{q,p}\left(  \Omega_{f}\right)  e^{\left\langle v,u\right\rangle
}F\left(  s_{f}\right)  \right)  ^{2}\ .
\end{equation}

\end{document}